\documentclass[a4paper,11pt,oneside,notitlepage]{article}
\usepackage{times}
\usepackage{lmodern}
\usepackage[T1]{fontenc}
\usepackage[utf8]{inputenc}
\usepackage[english]{babel}
\usepackage{amssymb}
\usepackage{amsmath}
\usepackage{mathtools} 
\usepackage{amsthm}
\usepackage{amsfonts, mathrsfs}
\usepackage{latexsym}
\usepackage[pdftex]{color, graphicx}
\usepackage{makeidx}
\usepackage{sidecap}
\usepackage{verbatim}
\usepackage{geometry}
\usepackage{subfig}
\usepackage{ulem}
\usepackage{tikz,float}
\usetikzlibrary{shapes,arrows,shadows}
\usepackage[displaymath,mathlines, running]{lineno} 
\usepackage{comment} 
\usepackage[colorinlistoftodos]{todonotes}
\usepackage{esint} 
\usepackage{pdfsync}

\usepackage{authblk}

\newtheorem{thm}{Theorem}[section]
\newtheorem{prop}[thm]{Proposition}
\newtheorem{lemma}[thm]{Lemma}

\newtheorem{definition}[thm]{Definition}
\newtheorem{remark}[thm]{Remark}

\def\R{\mathbb{R}}
\def\N{\mathbb{N}}
\def\Z{\mathbb{Z}}
\def\e{\varepsilon}
\def\loc{{\rm loc}}
\def\hom{{\rm hom}}

\def\esssup{\underset{x\in\Omega}{\mbox{ess-sup}}\hspace{0.03cm}}

\title{Homogenization of supremal functionals in {\color{blue} the} vectorial {\color{blue}case} (via  {\color{blue} $L^p$-approximation})} 
\author{Lorenza D'Elia, Michela Eleuteri, and Elvira Zappale}
\date{}

\AtEndDocument{\bigskip{\footnotesize%
  \textsc{Institute of Analysis and Scientific Computing, TU Wien, Wiedner Hauptstraße 8-10, 1040 Vienna, Austria} \par  
  \textit{E-mail address}: \texttt{lorenza.delia@tuwien.ac.at} \par
  \addvspace{\medskipamount}
  \textsc{Dipartimento di Scienze Fisiche Informatiche e Matematiche, Universit\`{a} degli Studi di Modena e Reggio Emilia, via Campi 213/b, 41125 Modena, Italy} \par  
  \textit{E-mail address}: \texttt{michela.eleuteri@unimore.it} \par
  \addvspace{\medskipamount}
  \textsc{ Dipartimento di Scienze di Base e Applicate per l’Ingegneria, Sapienza - Universit\`{a} di Roma,
via Antonio Scarpa, 16, 00161 Roma, Italy} \par
  \textit{E-mail address}: \texttt{elvira.zappale@uniroma1.it}
}}

\begin{document}

\maketitle

\begin{abstract}  We propose a homogenized supremal functional rigorously derived via \color{blue}$L^p$-approximation \color{black} by functionals of the type $\esssup f\left(\frac{x}{\varepsilon}, Du\right)$, when $\Omega$ is a bounded open set of $\mathbb R^n$ and $u\in W^{1,\infty}(\Omega;\mathbb R^d)$. The homogenized functional is also deduced directly in the case where the sublevel sets of $f(x,\cdot)$ satisfy suitable convexity properties, as a corollary of homogenization results dealing with pointwise gradient constrained integral functionals.
\color{black}\\
\medskip

{\scriptsize
}
\end{abstract}

\section{Introduction and statement of the problem}

\hspace{0.3cm} In many applications, such as dielectric breakdown for composite conductors (see \cite{GNP01}) or to model plasticity problems (see \cite{KL99}), it is customary to consider variational problems where the energies involved are not integral and where the relevant quantities that come into play do not express a mean property. Minimization problems therefore involve functionals of the type
\begin{equation}\label{pb01}
F(u) = \esssup f(x, u(x), D u(x)),
\end{equation}
where $\Omega\subset \mathbb R^n$ is a bounded open set, $u\in W^{1,\infty}(\Omega;\mathbb R^d)$, and $f: \Omega \times \mathbb R^d \times \mathbb R^{d\times n} \to \mathbb R$ is a suitable Borel function.
\color{black}
\textcolor{blue}{In view of} paper \cite{ABP02}, these problems have been addressed \textcolor{blue}{through the study of} {\it supremal functionals.}
It is worth to emphasize that this type of functionals can be seen as the counterpart of \textcolor{blue}{integral functionals, which give rise to} the Euler-Lagrange ({\it Euler-Aronsson}) equations (or systems) of $\Delta_\infty$-type (see the pioneering papers \cite{A1, A2, A3, A4}) and the more recent contributions \cite{AK, CKM}, among a wide literature.
\\
\par On the other hand, for {\color{blue} minimization problems involving} \eqref{pb01}, it is important to consider the pointwise behavior of the energy density also on very small sets; therefore it is relevant to ask whether this class of functionals is stable under $\Gamma$-convergence in $L^{\infty}.$

A first answer to this question has been derived in  \cite{BGP04} in the framework of homogenization, a theory establishing the macroscopic behaviour of materials with highly heterogeneous microstructure. Among the vast literature,  we single out the book \cite{BD} for an overview of homogenization of integral energies and the contributions \cite{BCPD21, BFL00, DF16} and references therein for some recent developments.\\  
In \cite{BGP04}, the authors consider the following \textcolor{blue}{family} of functionals
\[
F_\e(u):= \esssup f\left({x\over \e},  \nabla u(x)\right),
\]
where \textcolor{blue}{$\varepsilon > 0$,} $\Omega$ is a bounded set of $\mathbb{R}^n,$ $u \in W^{1,\infty}(\Omega)$ and the function $f$ is periodic in the first variable. The purpose of this contribution has been to replace this highly oscillating functional, as $\varepsilon$ goes to zero, with a simpler functional $F_{\rm hom}$, the homogenized functional, which is able to capture the relevant features of the family $\{F_\e\}_\e$.
Their work has been inspired by the papers \cite{CDP03, GNP01}: the energy density of the homogenized functional can be represented by means of a cell-problem formula
obtained by an approximation technique, thanks
to the fact that the \color{blue}$L^p$-\color{black}approximation described in the already mentioned contributions (consisting in approximating the $L^\infty$ functionals (possibly depending on $\varepsilon$) by an $L^p$-norm type ones, i.e. $(\int_\Omega f^p(\tfrac{x}{\varepsilon},\nabla u(x)) dx)^{\frac{1}{p}}$) holds true also in this inhomogeneous setting. Furthermore,  they obtain that the homogenization (i.e. a suitable variational limit as $\varepsilon \to 0$) and the \color{blue}$L^p$-\color{black}approximation (i.e. a variational limit as $p \to +\infty$) commute.\\
\color{black}
\par
In the present paper, we would like to complement the results in \cite{BGP04} (see also \cite{BPZ} and \cite{Z}) by considering homogenization for supremal functionals in the {\it vectorial {\color{blue} case}} by means of \color{blue}a $L^p$-\color{black}approximation.  More precisely,

let $\Omega$ be a bounded  open set of $\R^n$ with Lipschitz boundary. For any $\e>0$, we introduce the supremal functional defined on  $W^{1,\infty}(\Omega; \hspace{0.03cm}\R^d)$ by
   \begin{equation}
   \label{def:supfunctional}
   F_\e(u):= \esssup f\left({x\over \e}, Du(x)\right),
   \end{equation}
where the supremal integrand $f:\R^n\times \R^{d\times n}\to[0, +\infty)$ is a Borel function, $1$-periodic in the first variable which satisfies the following growth conditions: there exist  two positive constants $\alpha$ and $\beta$ such that 
          \begin{equation}
          \label{growthconditions}
          \alpha|Z|\leq f(x, Z)\leq \beta(|Z|+1),
          \end{equation}  
       for any $Z\in\R^{d\times n}$ and for every  $x\in\R^n$. %
       \color{blue} We observe that the coercivity condition in \eqref{growthconditions} can be relaxed to assume $\alpha |Z|-\frac{1}{\theta}\leq f(x,Z),$ with $\theta \in (0,+\infty)$, with some technical modifications in the arguments below, but we prefer to consider this homogeneous version for the readers' convenience since the current one is the version present in the literature we are referring to in the paper. \color{black}  
Our aim is to provide a homogenized version of the family $\{F_\e\}_\varepsilon$. To that end, we propose two strategies: one in terms of a direct relaxation procedure (\color{blue} under conditions which are natural in some model case, e.g. $f: \mathbb R^n \times \mathbb R^{d\times n} \ni(x,Z) \mapsto a(x)\Psi(Z)$ with 
\begin{align}\label{Psidef}
\Psi(Z):= \left\{\begin{array}{ll} 1 & \hbox{ if } |Z| \leq 1,\\
(1+\sqrt{||Z|-1|}+ ||Z|-1|) &\hbox{ if } |Z| >1,
\end{array}
\right.
\end{align} and $a \in C(\mathbb R^n), a \geq C, C>0$,  and $1$-periodic, i.e. $a(x)= a(x+ e_i)$, $\{e_i\}_{i=1,\dots, n}$ being the canonical basis of $\mathbb R^n$\color{black}), the other strategy, pursued in a  wider generality, consisting of an approximation approach by means of an \color{blue}$L^p$-\color{black}approximation.  More precisely, for any fixed $\e>0$  and $p >1$, \color{black} we introduce the double indexed \color{black} functional $F_{p,\e}: L^p(\Omega;\hspace{0.03cm}\R^d)\to \R\cup\{+\infty\}$ defined by  
   \begin{equation}
   \label{def:Functionalpeps}
   F_{p,\e}(u):= \begin{dcases}
   \left(\int_{\Omega} f^p\left({x\over \e}, Du(x) \right)dx \right)^{1/p} &\mbox{if } u\in W^{1,p}(\Omega;\hspace{0.03cm}\R^d),\\
   &\\
   +\infty&\mbox{otherwise in } L^p(\Omega;\mathbb R^d).
   \end{dcases}
   \end{equation}
Note that, in view of \eqref{growthconditions}, for any fixed $1<p<+\infty$, the density function $f^p(x, Z)$ satisfies the standard $p$-growth conditions in $W^{1,p}(\Omega; \hspace{0.03cm}\R^d)$, i.e., 
\begin{equation}
\notag
\alpha^p |Z|^p \leq f^p(x,Z) \leq \gamma (1+|Z|^p),
\end{equation}
for some constant $\gamma>0$ (\color{blue} e.g. $\gamma:= 2^p \beta^p$), \color{black}  for every $Z \in \mathbb R^{d \times n}$ and every $x \in \R^n$.  Applying standard results (see, e.g., \cite[Chapter 14]{BD} and \cite[Proposition 6.16 and Chapter 24]{DM93}), 
the functionals  $\{F_{p,\e}\}_{p,\varepsilon}$ $\Gamma(L^p)$-converges, as $\e\rightarrow 0$, to the functional $F^\hom_p$, i.e.
\begin{equation}\label{GammaLp}
\inf\{\liminf_{\varepsilon \to 0} F_{p,\varepsilon}(u_\varepsilon) : u_\varepsilon \to u \hbox{ in } L^p(\Omega;\mathbb R^d)\} = F^\hom_p(u),
\end{equation}
where
     \begin{equation}
     \label{def:Fhomp}
     F^\hom_p(u)\coloneqq
     \begin{dcases}
         \left(\int_{\Omega} f^\hom_p(Du(x))dx \right)^{1/p} &\hbox{ if }u \in W^{1,p}(\Omega;\mathbb R^d),\\
         &\\
         + \infty &\hbox{ otherwise in }L^p(\Omega;\mathbb R^d),
     \end{dcases}
     \end{equation}
where the effective energy density $f^\hom_p: \R^{d\times n}\to [0, +\infty)$ is characterized by the  so-called {\it asymptotic  homogenization formula }
   \begin{equation}
   \label{def:fhomp}
   f^\hom_p(Z) :=\lim_{T\to\infty} {1\over T^n} \inf\biggl\{\int_{TY} f^p(x, Z + Du(x))dx\hspace{0.03cm}:\hspace{0.03cm} u\in W^{1,p}_{0} (TY;\hspace{0.03cm} \R^d)  \biggr\}. 
   \end{equation} 
In addition, in light of \cite[Remark 14.6]{BD}, $f^\hom_p$ also satisfies the formula
   \begin{equation}
   \label{fhompdoubleinf}
   f^\hom_p(Z)=\inf_{j\in\mathbb{N}}\inf\left\{{1\over j^n} \int_{jY} f^p(x, Z+ Du(x))dx\hspace{0.03cm}:\hspace{0.03cm} u\in W^{1,p}_{\#} (jY;\hspace{0.03cm}\R^d) \right\},
   \end{equation}      
and in \eqref{def:fhomp}, one can replace $W^{1,p}_{0}(TY;\hspace{0.03cm}\R^d)$ with $W^{1,p}_{\#}(TY;\hspace{0.03cm}\R^d),$ where $\mathbb R \ni T >0,$ and $Y=(0,1)^n$ (we refer to Section \ref{not&pre} for \textcolor{blue}{the} adopted notation). We would like to \textcolor{blue}{stress that} the energy density $f$ is assumed to be a Borel function since we drop the quasiconvexity assumption in the second variable.\color{black}\\

\par At this point, we take the limit as $p\to +\infty$ of the functionals in \eqref{def:Fhomp}. In other words, we start considering first the asymptotics as the homogenization parameter vanishes in the sense of $\Gamma$-convergence with respect to the $L^p$ strong convergence, then we study the $\Gamma$-limit in the \color{blue}$L^p$-approximation \color{black} sense. This is done in Section \ref{tre} and the main result in this direction is Theorem \ref{Theorem3.6}, where it is proved that   \begin{equation}
   \notag   \Gamma(L^\infty)\mbox{-}\lim_{p\to+\infty} F^{\color{blue}\hom}_p(u)= F^\hom(u):=  \underset{x\in \Omega}{\mbox{{\rm ess-sup}}}\hspace{0.1cm} \widetilde{f}_\hom (Du(x)),
   \end{equation} for every $u \in W^{1, \infty}(\Omega; \mathbb{R}^d).$ 
   In fact, this latter functional can be considered as the homogenized version of \eqref{def:supfunctional}. \color{black} 
   We also emphasize that the effective energy density $\widetilde{f}_\hom$ is characterized through an asymptotic formula
   \begin{equation}\label{tildefhom}
   \widetilde{f}_\hom(Z):=\inf_{j\in\mathbb{N}} \inf \left\{ \underset{x\in jY}{\mbox{{\rm ess}-{\rm sup}}}\hspace{0.03cm} f_\infty(x, Z+Du(x)) \hspace{0.03cm}:\hspace{0.03cm} u\in W^{1,\infty}_\# (jY; \hspace{0.03cm}\R^d) \right\},
   \end{equation}
  \textcolor{blue}{where} 
   \[
   f_\infty(x, Z) := \sup_{k\geq 1} (\widetilde{f^k})^{1/k} (x, Z), 
    \]
for a.e. $x\in\Omega$ and for every $Z\in\mathbb R^{d\times n}$,  where $\widetilde{f^k}$ denotes the relaxed energy density of $f^k$, ($k \in \mathbb N$) given by \cite[Theorem 4.4.1]{Buttazzo}. In Section \ref{quattro}, we show that if the function $f(x,\cdot)$ is level convex and upper semicontinuous, then \eqref{tildefhom} turns into a cell formula (see \eqref{cellformula}). It is also worth recalling that in contrast with convexity, level convexity does not guarantee any semicontinuity (see \cite{RZ2}).
\\
\par 
\color{blue} We do not explicitly present any proof of the fact that, for every $\e>0$,  the $\Gamma(L^\infty)\mbox{-}\lim_{p\to +\infty} F_{p,\e}= F_\varepsilon(u)$ for every $u \in W^{1,\infty}(\Omega;\mathbb R^d)$, since the set of assumptions on our density $f$ in \eqref{def:supfunctional} and hence in \eqref{def:Functionalpeps}  (see Section 6 for more details) fits in the framework of Theorem \ref{curlcase2} in Appendix, i.e., see \cite[Theorem 2.2]{PZ20}, where this result has been proven in a wider generality.
\color{black}\\
\par The last step consists in proving a homogenization result for the family of energies $\{F_\e\}_{\e}$ given by \eqref{def:supfunctional}: this is done in Section \ref{cinque} under suitable technical assumptions. Hence, we can summarize our results by means of the following diagram, observing that the vertical right-hand side arrow \color{blue} (which guarantees the commutativity of the diagram) \color{black}is obtained by Theorem \ref{Theorem5.1}: 
\color{black}

\begin{figure}[ht]
        \centering
        \includegraphics[width=7cm]{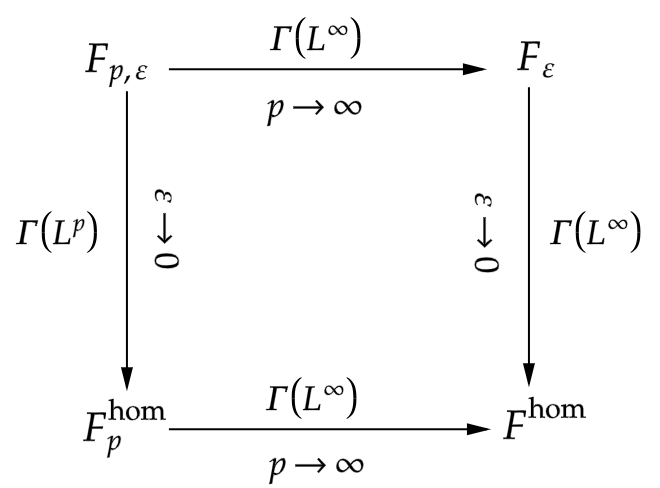}
          \label{Fig:comdiagram}
\end{figure}

\noindent More precisely,  we are able to prove that, for any $u\in W^{1,\infty}(\Omega; \hspace{0.03cm}\R^d)$,  
              \begin{equation}
               \notag
               \Gamma(L^\infty)\mbox{-}\lim_{\e\to 0} F_\e(u) = F^{\color{blue}\hom} (u).
           \end{equation}
We stress that the proof of the $\Gamma$-$\liminf$ inequality relies on the results obtained by 
 the $L^p$-approximation. Instead, the proof of the $\Gamma$-$\limsup$ inequality is more involved and it is based on a generalization of \cite[Chapter 12]{CDA02} to the vectorial case, requiring the convexity of $\Omega$ (strong-starshapedness {\color{blue} with Lipschtiz boundary} would suffice) and technical assumptions on the sublevel sets of the density $f(x,\cdot)$. 
 \color{black} Indeed, in Section \ref{secHomunb} we present a homogenization result, via $\Gamma$-convergence, of the family of vector-valued unbounded functionals of the form
    \[
        G_\e(u) := 
        \begin{dcases}
            \int_{\Omega} g\left({x\over \e}, Du(x)\right)dx & \quad \mbox{if } u\in W^{1,q}_{\loc} (\R^n; \hspace{0.03cm} \R^d)\cap L^\infty_{\loc} (\R^n; \hspace{0.03cm} \R^d), \\
            &\\
            +\infty & \mbox{otherwise},
        \end{dcases}        
    \]
where $\Omega$ is a bounded convex set (or strongly star-shaped \color{blue} with Lipschtiz boundary\color{black}), \color{black} $q\in [1, +\infty]$ and the energy density $g: \R^n\times \R^{d \times n} \to [0, +\infty]$  is ${\mathcal L}(\R^n)\otimes {\mathcal B}(\R^{d \times n})$-measurable, $Y$-periodic in the first variable and convex in the second one,  with the effective domain ${\rm dom}(g(x, \cdot))$  of $g(x,\cdot)$ containing a fixed cube $Q$, uniformly in $x$. \color{black}
 The main result of Section \ref{secHomunb}, namely the homogenization Theorem \ref{Proposition 12.6.1}, is applied, in Theorem \ref{Theorem5.1}, to the (integral of the) indicator function 
   \begin{equation*}
      1_{C(x)}(Z)\coloneqq
     \begin{dcases}
           0 & \hbox{ if } Z \in C(x), \\
           +\infty & \hbox{ otherwise},
      \end{dcases}
   \end{equation*}
 where $C(x)\subset \mathbb R^{d\times n}$ is a bounded Borel convex set, with no empty interior satisfying suitable hypotheses, as those described above, uniformly with respect to $x \in \Omega$. This result provides \color{black} also the limsup inequality. The main difficulty is to extend the results in \cite[Chapter 12]{CDA02} to \color{black} the vectorial setting. In this respect, we extended the scalar techniques in \cite{CDA02}, which, in turn, required, for instance, the need to impose further technical assumptions (see Remark \ref{remark:existenceofcube}), which, on the other hand, are similar to those present in literature in the context of relaxation for unbounded functionals depending on vector-valued fields, see, for instance, \cite{AH, W1, W2}.

On the other hand, the results contained in Section 5 are per se interesting, since they provide a vectorial extension of the results proven by \cite{AHCM, AHM2, CCGDA} and to the bibliography contained therein, we also refer to the recent contributions to integral representation problem in the unbounded vectorial setting contained in \cite{AHM, DG}, and to the results in the mechanical framework \cite{CK1, CK2}.

\color{black}

  \section{Notation and preliminary results}\label{not&pre}

\subsection{General notations and basic facts}

The present section is devoted to the introduction of the general notation
and the basic facts that we use throughout the paper.

\begin{itemize}

\item Let $X$ be a set. For every $S \subseteq X$, we denote by $\chi_S$ the characteristic
function of $S$ defined by
\[
\chi_S(x) = \left \{
\begin{array}{lll}
\!\!\!\!\!\! & 1 \qquad \textnormal{if $x \in S,$}\\[3mm]
\!\!\!\!\!\! & 0 \qquad \textnormal{if $x \in X \setminus S.$}
\end{array}
\right.
\]
{\color{blue} For every $S\subseteq X$, the indicator function $1_S$ is given by
   \begin{equation*}
       1_{S}(x)\coloneqq
       \begin{dcases}
            0 & \hbox{ if } x \in S, \\
           +\infty & \hbox{ if } x \in X \setminus S;
       \end{dcases}
   \end{equation*}

}
\item For any $x_0 \in \mathbb{R}^n$  and $r \in (0, + \infty),$ we denote by $B_r(x_0)$ the open ball of $\mathbb{R}^n$ centered in $x_0$ and with radius $r$ and by $Q_r(x_0)$ the open cube centred in $x_0$ having sidelength $r;$
\item For any $x_0 \in \mathbb{R}^n,$ $S, {\color{blue} E} \subseteq \mathbb{R}^n$ and $r \in (0, + \infty),$ we set
\begin{eqnarray*}
&& {\rm dist}(x_0, S) = \inf \{|x_0 - x|: x \in S\},\\[3mm]
&& {\color{blue}{\rm dist}(E, S) = \inf \{|x - y|: x \in  E, x\in S \},}\\[3mm]
&& S_r^- = \{x \in S: {\rm dist}(x, \partial S) > r\}, \\[3mm]
&& S_r^+ = \{x \in \mathbb{R}^n: {\rm dist}(x, S) < r\};
\end{eqnarray*}
\item We denote by $Y$ the unit cell in $\R^n$, i.e., $Y = (0,1)^n$ and by $Q$ a generic cube in $\mathbb R^{d\times n}$ which will be specified case by case; 
\item For every $N \in \mathbb N$, by $\mathcal L^N$ we denote the $N$-dimensional Lebesgue measure;
\item We say that a subset of $\mathbb{R}^n$ is a polyhedral set if it can be expressed
as the intersection of a finite number of closed half-spaces;
\item $\R^{d\times n}$ denotes the space of $d\times n$ real matrices;
\item Given $Z\in\R^{d\times n}$, the norm $|Z|$ is defined by $|Z|:=\sum_{i=1}^{d}|Z_i|$, where $Z_i$ is the $i$-th row of the matrix $Z$ and $|Z_i|$ is its Euclidean norm;
\item For every $O \subset \mathbb R^n$, by $X(O;\mathbb R^d)$ we denote a space of functions defined in $O$ with values in $\mathbb R^d$. $X_\#(KY; \mathbb{R}^d)$ denotes the space of $K$-periodic functions in $X_{\rm loc}(\mathbb{R}^n; \mathbb{R}^d),$ ($K \in (0,+\infty)$) \color{black} where 
\[
X_{\rm loc}(O; \mathbb{R}^d) = \{u: O \rightarrow \mathbb{R}^d: u \in X(U; \mathbb{R}^d) \,\,\, \textnormal{for all open sets} \,\,\,  U \subset \subset O \}.
\]
For Sobolev spaces, taking into account the notion of traces given in \cite[Section 4]{CDA02},  and for any cube $D\subset \mathbb R^n$, we denote, for every $p\in [1,+\infty]$, by
$W^{1,p}_{\#}({\color{blue} D};\mathbb R^d)$ the set
$\{v \in W^{1,p}({\color{blue} D};\mathbb R^d): \gamma_{\color{blue} D} v$ takes the same value on the opposite faces of ${\color{blue} D} \}$ {\color{blue} where $\gamma_D$ is the trace operator (for the precise definition, see \cite[Theorem 4.2.12]{CDA02} )};
\item for every open subset $\Omega \subset \mathbb R^n$,  $\textcolor{blue}{\mathcal{T}}(\Omega)$ denotes the family of open subsets of $\Omega$, and $\textcolor{blue}{\mathcal{T}}_0(\Omega)$ denotes those which are bounded. {\color{blue} In case $\Omega$ is $\R^n$, we denote with $\mathcal O$ the family of open subsets of $\R^n$};

\item For any function $f: \Omega\times \mathbb \R^{d \times n}$, by $f^{\rm ls}$ we denote the lower semicontinuous envelope of $f(x,\cdot)$ (i.e., with respect to the second variable);

\item For every $E \color{blue}\subset\color{black} \mathbb{R}^n,$ every function $u$ on $E$ and $x_0 \in \mathbb{R}^n,$ we define the {\it translate of $u$} as
\[
T[x_0]u : x \in E - x_0 \mapsto u(x + x_0).
\]
Let $\Phi: \mathcal{O} \times {\color{blue} X(\Omega; \R^d)} \rightarrow (- \infty, + \infty].$ We say that $\Phi$ is {\it translation invariant} if
\[
\Phi(\Omega - x_0, T[x_0]u) = \Phi(\Omega, u) \qquad \textnormal{for every $\Omega \in \mathcal{O}, x_0 \in \mathbb{R}^n, u \in {\color{blue} X(\Omega; \R^d)};$}
\]
\item For any $A \subset \mathbb R^n$,  we denote by ${\rm int}(A)$ the relative interior (relative to the smallest affine subspace which contains $A$) of $A$;

\item We denote by $\textcolor{blue}{A_{\rm pw}}(\mathbb{R}^n; \mathbb{R}^d)$ the set of the piecewise affine functions $u :\mathbb{R}^n \rightarrow \mathbb{R}^d$, i.e., $u$ is a continuous function 
such that 
\begin{equation}\label{aff-u-Zi}u(x) :=\sum_{i=1}^m (u_{Z_i}(x)+ c_i) \chi_{P_i}(x)\quad  \hbox{ for every }x \in \bigcup_{i=1}^m {\rm int}(P_i), 
\end{equation} with $m \in \mathbb N,$ where for any $i=1,\dots, m$, $u_{Z_i}$ is a linear function given by $u_{Z_i}(x)\coloneqq Z_ix$, with $ Z_1,\dots, Z_m\in \mathbb R^{d\times n}, c_1, \dots,c_m \in \mathbb R^d$,  and $P_1,\dots, P_m$
polyhedral
sets with pairwise disjoint nonempty interiors such that $\bigcup_{i=1}^m P_i=\mathbb R^n$;

\item Given $f:\mathbb R^{d \times n} \to [0,+\infty]$, by ${\rm dom} f$ we denote the effective domain of $f$, i.e.,
$$
{\rm dom }f:=\{Z \in \mathbb R^{d \times n}: f(Z)< +\infty\};$$

\end{itemize}

\subsection{$\Gamma$-convergence}

\noindent For an exhaustive introduction to $\Gamma$-convergence, we refer to \cite{DM93}. 
We only recall the sequential characterization of the $\Gamma$-limit when $X$ is a metric space.

\begin{prop}[{\cite[Proposition 8.1]{DM93}}] 
Let $X$ be a metric  space and let $\varphi_{k}: X \to \mathbb R \cup
\{\pm \infty\}$, for every $k\in \N$.
 Then  $\{\varphi_k\}_{k}$ $\Gamma$-converges to $\varphi$ with respect to the strong topology of  $X$ (and we write $\displaystyle  \Gamma(X)\hbox{-}\lim_{k\to +\infty}\varphi_{k}=\varphi$)  if and only if
\begin{itemize}
\item [{\rm (i)}] {\rm($\Gamma$-liminf inequality)} for every $x\in X$ and for every sequence $\{x_{k}\}_k$ converging
to $x$, it is
$$  \varphi(x)\le \liminf_{k\to +\infty} \varphi_{k}(x_{k});$$
\item [{\rm (ii)}]{\rm ($\Gamma$-limsup inequality)} for every $x\in X,$ there exists a sequence $\{x_{k}\}_k$  converging
to $x \in X$ such that
$$  \varphi(x)=\lim_{k\to +\infty} \varphi_{k}(x_{k}).$$
\end{itemize}
\end{prop}
We recall that the  $\Gamma\hbox{-}\lim_{k\to +\infty}\varphi_{k}$ is lower semicontinuous on $X$ (see \cite[Proposition 6.8]{DM93}).

\color{blue} In what follows we deal with the strong convergence in $L^\infty$, consequently the following definition will be assumed to be \color{black} the notion of $\Gamma$-convergence for families of functionals $\{\varphi_{\varepsilon}\}_\varepsilon$ as $\mathbb R^+ \ni \varepsilon \to 0$.
We also refer to \cite[Propositions 3.2.3, 3.2.4, 3.2.6]{CDA02} to see that $\Gamma^-$-convergence (with respect to a general topology $\tau$) (see Definition 3.2.1, therein), well suited for families of functionals reduces, in the context, subject of our subsequent analysis, \color{black} to the definition below

\begin{definition}
We say that a family $\{\varphi_{\varepsilon}\}_\e$ $\Gamma$-converges to $\varphi$, with respect to the metric in $X$ as $\varepsilon\to 0^+$,
if $\{\varphi_{{\varepsilon}_{k}}\}_k$ $\Gamma$-converges to $\varphi$ for all sequences $\{{\varepsilon}_{k}\}_k$ of positive numbers converging to $0$ as $k\to +\infty$.
\end{definition}
\smallskip

Furthermore, we recall the Urysohn property for the $\Gamma$-convergence of sequences of functionals defined on metric spaces $X$ (or at least satisfying the first axiom of countability).
\begin{prop}
Under the above assumptions on $X$, let $\lambda \in [-\infty,+\infty]$, then
\begin{equation}
\notag
\lambda = \Gamma\mbox{-}\lim_{k \rightarrow + \infty} \varphi_k({\color{blue} x})
\end{equation}
if and only if
\begin{equation}
\notag
\textnormal{$\forall \{k_h\} \subseteq \mathbb{N}$ 
strictly increasing, $\exists \{k_{h_j}\} \subseteq \{k_h\}$ such that $\lambda = \Gamma \mbox{-}\lim_{j \rightarrow + \infty} \varphi_{k_{h_j}}({\color{blue} x})$}.
\end{equation}
\end{prop}

\color{black}

\section{The $\boldsymbol{L^p}$-approximation}

\label{tre} 

In this section, after having recalled the homogenization of the functionals $\{F_{p, \e}\}_{p,\varepsilon}$ in \eqref{GammaLp}, as $\varepsilon \to 0$, leading to the functional \eqref{def:Fhomp}, we take the limit as $p\to +\infty$. {\color{blue} More precisely}, we start considering first the asymptotics as the homogenization parameter vanishes in the sense of $\Gamma$-convergence with respect to the $L^p$ strong convergence, then we study the $\Gamma$-limit in the \color{blue}$L^p$-approximation \color{black} sense, i.e., as $p\to +\infty$. To that end, we combine techniques used in \cite[Lemma 3.2]{BGP04}  and \cite[Theorem 2.2]{PZ20}. In particular, we recall the following definition, introduced in \cite{BJW04}, and proved to be necessary and sufficient for the lower semicontinuity of supremal functionals.
\begin{definition}
    \label{sMqcx}
    Let $f:\mathbb R^{d \times n}\to \R$ be a Borel measurable function. We say that $f$ is {\it strong Morrey quasiconvex} if 
$$\forall\ \varepsilon>0\ \forall\ Z\in \mathbb{R}^{d\times n}\ \forall\ K>0\ \exists\ \delta=\delta(\varepsilon, K,Z)>0:$$
$$\left.\begin{array}{l}\varphi\in W^{1,\infty}(Y;\mathbb{R}^d)\vspace{0.2cm}\\ ||D\varphi||_{L^\infty(Y;\mathbb{R}^{d\times n})}\le K\vspace{0.2cm}\\ \max_{x\in\partial Y}|\varphi(x)|\le \delta\end{array}\right\}\Longrightarrow f(Z)\le \underset{x\in Y}{\mbox{\rm ess-sup}}\hspace{0.03cm} f(Z+D\varphi(x))+\varepsilon.$$
\end{definition}

Let $\Omega$ {\color{blue} be a bounded open set of $\R^n$} and let $f:\Omega \times \mathbb R^{d\times n}\to [0,+\infty]$,  be a $\mathcal L(\Omega)\otimes \mathcal B(\mathbb R^{d \times n})$- measurable function. Following \cite{PZ20}
we set
    \begin{equation}
        \label{def:finfinity}
        f_\infty(x, Z) := \sup_{k\geq 1} (\widetilde{f^k})^{1/k} (x, Z), 
    \end{equation}
for a.e. $x\in\Omega$ and for every $Z\in\mathbb R^{d\times n}$,  where $\widetilde{f^k}$ denotes the relaxed energy density of $f^k$, ($k \in \mathbb N$) given by \cite[Theorem 4.4.1]{Buttazzo}, which is a Carath\'eodory function, quasiconvex in the second variable.

\color{blue} We remark that, in general, we cannot conclude that  $\widetilde{f^k}$ coincides with the greatest quasiconvex minorant of $f^k(x,\cdot)$, as observed in \cite[Remark 4.4.5]{Buttazzo}, see also Remark \ref{cexButtazzo} below for further comments. \color{black} 

  \begin{lemma}
  \label{lemma:recoveryseq}
  Let $f:\R^n\times\R^{d\times n}\to [0,+\infty]$ be a Borel function, $1$-periodic in the first variable, satisfying \eqref{growthconditions}. Let $f^\hom_p$ be the homogenized energy density described by \eqref{fhompdoubleinf}. Then, for any $Z\in\R^{d\times n}$, it holds 
         \begin{equation}
         \label{limfhomp}
         \lim_{p\to+\infty} (f^\hom_p)^{1/p}(Z) =\widetilde{f}_\hom(Z),
         \end{equation}
      where $\widetilde{f}_\hom$ is given by the following formula
         \begin{equation}
         \label{characterizationftildehom}
         \widetilde{f}_\hom(Z):=\inf_{j\in\mathbb{N}} \inf \left\{ \underset{x\in jY}{\mbox{{\rm ess}-{\rm sup}}}\hspace{0.03cm} f_\infty(x, Z+Du(x)) \hspace{0.03cm}:\hspace{0.03cm} u\in W^{1,\infty}_\# (jY; \hspace{0.03cm}\R^d) \right\},
         \end{equation}
         for any $Z \in\R^{d\times n}$, where $f_\infty$ is given by \eqref{def:finfinity}.
 \end{lemma}  
\begin{proof}
 To prove \eqref{limfhomp}, we show that $(f^\hom_p)^{1/p}$ is a non-decreasing function with respect to $p>1$, i.e., for $p_1<p_2$, we have that, for any $Z \in\R^{d\times n}$,
    \begin{equation}
    \label{fhomincreasingsequence}
    (f^\hom_{p_1})^{1/p_1}(Z)\leq(f^\hom_{p_2})^{1/p_2}(Z).
    \end{equation}
Fix $j\in\mathbb{N}$ and $Z\in\R^{d\times n}$. Since $f^{p_2}(\cdot, Z+Du(\cdot))\in L^1(jY)$, it follows that $f^{{\color{blue}p_1}}(\cdot, Z+Du(\cdot))\in L^{p_2/p_1}(jY)$, so that an application of the H\"{o}lder inequality yields to 
    \begin{equation}
    \notag
    \int_{jY} f^{p_1}\left(x, Z+ Du(x)\right)dx\leq (j^n)^{1-p_1/p_2}\left(  \int_{jY} f^{p_2}\left(x, Z+ Du(x)\right)dx\right)^{p_1/p_2}.
    \end{equation}
Moreover, due to $W^{1, p_2}_{\#}(jY;\hspace{0.03cm}\R^d)\subset W^{1, p_1}_{\#}(jY;\hspace{0.03cm}\R^d)$, we deduce that 
   \begin{align}
   {1\over j^n}\inf&\left\{\int_{jY} f^{p_1}\left(x, Z+ Du(x)\right)dx \hspace{0.03cm}:\hspace{0.03cm} u\in W^{1, p_1}_{\#}(jY;\hspace{0.03cm}\R^d)    \right\}\notag\\
   & \leq{1\over j^n}\inf\left\{\int_{jY} f^{p_1}\left(x, Z+ Du(x)\right)dx \hspace{0.03cm}:\hspace{0.03cm} u\in W^{1, p_2}_{\#}(jY;\hspace{0.03cm}\R^d)    \right\}\notag\\
   &\leq {1\over (j^n)^{p_1/p_2}}\inf\left\{\left(  \int_{jY} f^{p_2}\left(x, Z+ Du(x)\right)dx\right)^{p_1/p_2} \hspace{0.03cm}:\hspace{0.03cm} u\in W^{1, p_2}_{\#}(jY;\hspace{0.03cm}\R^d)    \right\}\notag\\
   &=\left[{1\over j^n}\inf\left\{  \int_{jY} f^{p_2}\left(x, Z+ Du(x)\right)dx\hspace{0.03cm}:\hspace{0.03cm} u\in W^{1, p_2}_{\#}(jY;\hspace{0.03cm}\R^d)    \right\}  \right]^{p_1/p_2}\notag.
   \end{align}
Taking the infimum on $j\in\mathbb{N}$ and using \eqref{fhompdoubleinf}, it follows that
     \begin{align}
     & f^\hom_{p_1}(Z)\notag\\
     &\quad \leq \inf_{j\in\mathbb{N}} \left[{1\over j^n}\inf\left\{  \int_{jY} f^{p_2}\left(x, Z+ Du(x)\right)dx\hspace{0.03cm}:\hspace{0.03cm} u\in W^{1, p_2}_{\#}(jY;\hspace{0.03cm}\R^d)    \right\}  \right]^{p_1/p_2}\notag\\
     &\quad =\left[\inf_{j\in\mathbb{N}}  {1\over j^n}\inf\left\{  \int_{jY} f^{p_2}\left(x, Z+ Du(x)\right)dx\hspace{0.03cm}:\hspace{0.03cm} u\in W^{1, p_2}_{\#}(jY;\hspace{0.03cm}\R^d)    \right\}  \right]^{p_1/p_2}\notag\\
     &\quad =(f^\hom_{p_2})^{p_1/p_2}(Z)\notag,
     \end{align}
for any $Z \in\R^{d\times n},$ which implies \eqref{fhomincreasingsequence}. Hence, the limit as $p\to+\infty$ of $(f^\hom_p)^{1/p}$   exists and we denote it with $\widetilde{f}_\hom$. In other words, 
  \[
   \lim_{p\to+\infty} (f^\hom_p)^{1/p} (Z)= :\widetilde{f}_\hom(Z) \qquad\mbox{for any }Z\in\R^{d\times n}.
   \]
To conclude the proof, it remains to show the characterization \eqref{characterizationftildehom}.
First, note that,
thanks to growth condition \eqref{growthconditions} along with \cite[Proposition 6.11]{DM93}, for $p>n$, $f^\hom_p$ may be characterized through the relaxed energy density $\widetilde{f^p}$, i.e.,
    \begin{equation}
        \notag
        f^\hom_p(Z) := \inf_{j\in\mathbb{N}}\inf\left\{{1\over j^n} \int_{jY} \widetilde{f^p}(x, Z+Du(x))dx\hspace{0.03cm} : \hspace{0.03cm} u\in W^{1,p}_{\#}(jY;\hspace{0.03cm} \R^d)  \right\}.
    \end{equation}
Hence, for fixed $j\in\mathbb{N}$ and $Z\in\R^{d\times n}$ and for a given $u\in W^{1,\infty}_{\#} (jY; \hspace{0.03cm}\R^d)$,

we deduce that
   \begin{equation}
   \notag
   \int_{jY} \widetilde{f^p} (x, Z+ Du(x))dx\leq j^n \underset{x\in jY}{\mbox{ess-sup}}\hspace{0.03cm} \widetilde{f^p}(x, Z+Du(x)).
   \end{equation}
This, combined with the fact that $W^{1,\infty}_{\#} (jY; \hspace{0.03cm}\R^d) \subset W^{1,p}_{\#} (jY; \hspace{0.03cm}\R^d),$ implies that
    \begin{align}
    {1\over j^n}&\inf \left\{\int_{jY} \widetilde{f^p} (x, Z+ Du(x))dx \hspace{0.03cm}:\hspace{0.03cm} u\in W^{1,p}_{\#} (jY; \hspace{0.03cm}\R^d) \right\} \notag\\
    & \leq {1\over j^n}\inf \left\{\int_{jY} \widetilde{f^p}(x, Z+ Du(x))dx \hspace{0.03cm}:\hspace{0.03cm} u\in W^{1,\infty}_{\#} (jY; \hspace{0.03cm}\R^d) \right\} \notag\\
    &\leq \inf \left\{\underset{x\in jY}{\mbox{ess-sup}}\hspace{0.3cm} \widetilde{f^p}(x, Z+Du(x))  \hspace{0.03cm}:\hspace{0.03cm} u\in W^{1,\infty}_{\#} (jY; \hspace{0.03cm}\R^d) \right\}. \notag
    \end{align}
Taking the infimum on $j\in\mathbb{N}$, we get that
   \begin{align}
   f^\hom_p(Z) &= \inf_{j\in\mathbb{N}}\inf\left\{{1\over j^n} \int_{jY} \widetilde{f^p}(x, Z+ Du(x))dx\hspace{0.03cm}:\hspace{0.03cm} u\in W^{1,p}_{\#} (jY;\hspace{0.03cm}\R^d) \right\}\notag\\
   &\leq \inf_{j\in\mathbb{N}} \inf \left\{\underset{x\in jY}{\mbox{ess-sup}}\hspace{0.03cm} \widetilde{f^p}(x, Z+Du(x)) \hspace{0.03cm}:\hspace{0.03cm} u\in W^{1,\infty}_\# (jY; \hspace{0.03cm}\R^d) \right\}.\notag
   \end{align}
This implies that 
   \begin{align}
   (f^\hom_p(Z) )^{1/p} &\leq \inf_{j\in\mathbb{N}} \inf \left\{\left(\underset{x\in jY}{\mbox{ess-sup}}\hspace{0.1cm} \widetilde{f^p}(x, Z+Du(x))\right)^{1/p} \hspace{0.03cm}:\hspace{0.03cm} u\in W^{1,\infty}_\# (jY; \hspace{0.03cm}\R^d) \right\} \notag\\
   &= \inf_{j\in\mathbb{N}} \inf \left\{\underset{x\in jY}{\mbox{ess-sup}}\hspace{0.1cm} (\widetilde{f^p})^{1/p} (x, Z+Du(x))\hspace{0.03cm}:\hspace{0.03cm} u\in W^{1,\infty}_\# (jY; \hspace{0.03cm}\R^d) \right\}.\notag
   \end{align}
The same holds when $p\equiv p_k$ is an integer number, therefore it is possible to repeat the same arguments as before by considering
a divergent sequence $\{p_k\}_k$ as $k\to+\infty$. We therefore deduce that
   \begin{align}
    (f^\hom_{p_k}(Z) )^{1/p_k} \leq   \inf_{j\in\mathbb{N}} \inf \left\{\underset{x\in jY}{\mbox{ess-sup}}\hspace{0.1cm} (\widetilde{f^{p_k}})^{1/p_k} (x, Z+Du(x))\hspace{0.03cm}:\hspace{0.03cm} u\in W^{1,\infty}_\# (jY; \hspace{0.03cm}\R^d) \right\}.\label{ineq4}
   \end{align}
{\color{blue} Given a bounded open set $\Omega$ of $\R^n$,} we introduce a ${\mathcal L}^n(\Omega)\otimes {\cal B}(\mathbb R^{d\times n})$-measurable function $h_\infty: \Omega\times\R^{d\times n}\to [0, +\infty)$ defined by
    \begin{equation}
        \notag
        h_\infty (x, Z) := \sup_{k\in\mathbb{N}} \left(\widetilde{f^{p_k}} \right)^{1/p_k} (x, Z).
    \end{equation}
It is possible to show that, for any $u\in W^{1,\infty}(\Omega; \hspace{0.03cm} \R^d)$,
      \begin{equation}
        \label{eqhf}
        \underset{x\in jY}{\mbox{ess-sup}}\hspace{0.1cm} h_\infty(x, Du(x)) = \underset{x\in jY}{\mbox{ess-sup}}\hspace{0.1cm} f_\infty(x, Du(x)),
      \end{equation}
for a proof see \color{blue} the Appendix where the arguments taken from \cite[equations (61) and (64) in Theorem 2.2]{PZ20} are presented. \color{black} 
 Hence, from \eqref{ineq4}, it follows that
    \begin{align}
        (f^\hom_{p_k}(Z) )^{1/p_k} &\leq   \inf_{j\in\mathbb{N}} \inf \left\{\underset{x\in jY}{\mbox{ess-sup}}\hspace{0.1cm}h_\infty (x, Z+Du(x))\hspace{0.03cm}:\hspace{0.03cm} u\in W^{1,\infty}_\# (jY; \hspace{0.03cm}\R^d) \right\}\notag\\
        &=\inf_{j\in\mathbb{N}} \inf \left\{\underset{x\in jY}{\mbox{ess-sup}}\hspace{0.1cm}f_\infty (x, Z+Du(x))\hspace{0.03cm}:\hspace{0.03cm} u\in W^{1,\infty}_\# (jY; \hspace{0.03cm}\R^d) \right\}.\notag
    \end{align}
Finally, taking the limit as $k\to+\infty$, taking into account \eqref{limfhomp}, we conclude that
    \begin{align}\label{1st}
        \widetilde{f}_\hom(Z) \leq \inf_{j\in\mathbb{N}} \inf \left\{\underset{x\in jY}{\mbox{ess-sup}}\hspace{0.1cm}f_\infty (x, Z+Du(x))\hspace{0.03cm}:\hspace{0.03cm} u\in W^{1,\infty}_\# (jY; \hspace{0.03cm}\R^d) \right\}.
    \end{align}
\par Now, we show the reverse inequality, assuming, without loss of generality, that $\widetilde f_{\hom}$ in \eqref{limfhomp} is finite. To that end, we set
	\begin{equation}
	\label{varphi}
	\varphi(Z):= \inf_{j\in\mathbb{N}}\inf\left\{ \underset{x\in jY}{\mbox{ess-sup}}\hspace{0.1cm} f_\infty(x, Z+Du(x)) \hspace{0.03cm}:\hspace{0.03cm} u\in W^{1,\infty}_\# (jY; \hspace{0.03cm}\R^d) \right\},
	\end{equation}
	 and we fix $\delta > 0$. In view of the characterization of $f^\hom_p$ given by  \eqref{fhompdoubleinf}, we  deduce  that for any $p>1$ there exist $\overline{j}\in\mathbb{N}$ and $u_p\in W^{1,p}_{\#}(\overline{j}Y\hspace{0.03cm};\hspace{0.03cm}\R^d)$ such that

 \color{black}
	\begin{equation}
	\notag
	\left({1\over\overline{j}^n}\int_{\overline{j}Y}f^p(x, Z+Du_p(x))dx\right)^{1/p}\leq (f^\hom_p(Z))^{1/p} +\delta
	\end{equation}  
Using the growth conditions \eqref{growthconditions} as well as the triangular inequality for the $L^p$-norm, it follows that 
	\begin{align}
	\left({1\over\overline{j}^n}\int_{\overline{j}Y}
	|Du_p(x)|^pdx\right)^{1/p}&\leq
	\left({1\over\overline{j}^n}\int_{\overline{j}Y}
		|Du_p(x)+Z|^pdx\right)^{1/p} + \left({1\over\overline{j}^n}\int_{\overline{j}Y}
			|Z|^pdx\right)^{1/p}\notag\\
	&\leq \left({1\over\overline{j}^n}\int_{\overline{j}Y}f^p(x, Z+Du_p(x))dx\right)^{1/p}+|Z|\notag\\
 & \leq (f^\hom_p(Z))^{1/p} +\delta + |Z|\notag\\
	&\leq C(\delta, Z),\notag
	\end{align}
 where the latter constant does not depend on $p$.
 \color{black}
This, combined with the fact that, for any $p>q>n$,
we have that $L^p(\overline{j}Y\hspace{0.03cm};\hspace{0.03cm}\R^{d\times n} )\subset L^q(\overline{j}Y\hspace{0.03cm};\hspace{0.03cm}\R^{d\times n})$, we deduce that
	\begin{equation}
	\notag
	{1\over \overline{j}^n}\|Du_p\|_{{L^q(\overline{j}Y\hspace{0.03cm}; \R^{d\times n}})}\leq {1\over \overline{j}^n} \|Du_p\|_{L^p(\overline{j}Y\hspace{0.03cm}; \R^{d\times n})}\leq C.
	\end{equation}
 \color{black}
	It is not restrictive to assume that $u_p$ has zero average, so that due to Poincar\'{e}-Wirtinger's inequality, we conclude that the sequence $\{u_p\}_{p>q}$ is bounded in $W^{1,q}_{\#}(\overline{j}Y\hspace{0.03cm};\hspace{0.03cm}\R^d)$, for $q>n$. Therefore, up to subsequences, there exists $u_\infty$ such that $u_p$ weakly converges to $u_\infty$ in $W^{1,q}(\overline{j}Y\hspace{0.03cm};\hspace{0.03cm} \R^d)$. 
	Thanks to the compactness of embedding of Sobolev spaces (see, e.g., \cite[Theorem 9.16]{Bre10}), we conclude that $u_p$ uniformly converges to $u_\infty$ as $p\to +\infty$, since $q>n$. Moreover, $u_\infty\in W^{1,\infty}_{\#}(\overline{j}Y\hspace{0.03cm};\hspace{0.03cm}\R^d)$. Indeed, from the boundedness of $\{D u_p \}_{p>q}$ in $L^q(\overline{j}Y\hspace{0.03cm}; \R^{d\times n})$ combined with the $W^{1,q}$-weak convergence of $\{u_p \}_{p>q}$ to $u_\infty$ and the lower semicontinuity of the norm, it follows that 
	     \begin{equation}
	     \notag
	     \|D u_\infty\|_{L^q(\overline{j}Y\hspace{0.03cm};\R^{d\times n})}\leq \liminf_{p\to+\infty} \|D u_p\|_{L^q(\overline{j}Y\hspace{0.03cm}; \R^{d\times n})}\leq C.
	     \end{equation}  
	 Now, taking into account that the $L^\infty$ norm can be approximated from below by the $L^q$ norm, as $q \to +\infty$, computing this latter limit,  we conclude that $\|D u_\infty\|_{L^\infty(\overline{j}Y\hspace{0.03cm};\hspace{0.03cm} \R^{d\times n})}\leq C$. This together with the uniform convergence of $u_p$ to $u_\infty$ implies that $u_\infty\in W^{1,\infty}_{\#}(\overline{j}Y\hspace{0.03cm};\hspace{0.03cm}\R^d)$.
	 Hence, we have shown that, up to subsequences, there exists $u_\infty\in W^{1,\infty}_{\#}(\overline{j}Y\hspace{0.03cm};\hspace{0.03cm}\R^d)$ such that $u_p$ uniformly converges to $u_\infty$, as $p\to+\infty$.
	Recalling \eqref{varphi}, thanks to  \cite[Theorems 2.2]{PZ20}, it follows that 
	\begin{align}
	\varphi(Z)&\leq \underset{x\in \overline{j}Y}{\mbox{ess-sup}}\hspace{0.1cm} f_\infty(x, Z+Du_\infty(x))\notag\\
	&\leq \liminf_{p\to+\infty} \left({1\over \overline{j}^n} \int_{\overline{j}Y }f^p(x, Z+Du_p(x))dx \right)^{1/p}\notag\\
	&\leq \liminf_{p\to+\infty}(f^\hom_p(Z))^{1/p} + \delta\notag\\
	&= \widetilde{f}_\hom (Z) +\delta. \notag
	\end{align}
	In view of the arbitrariness of $\delta$, we obtain the reverse inequality of \eqref{1st}, as desired.
\end{proof}   

\color{black}

\begin{remark}\rm  Since $f^\hom_p$ is also characterized through  asymptotic formula 
   \begin{equation}
       \notag
       f^\hom_p(Z)=\limsup_{T\to+\infty} \inf{1\over T^n} \left\{\int_{TY}\widetilde{f^p}(x, Z+Du(x))dx\hspace{0.03cm}:\hspace{0.03cm} u\in W^{1,p}_{\#}( TY; \hspace{0.03cm}\R^d)  \right\}, 
   \end{equation}
we deduce that 
	\begin{equation}
	\label{upperboundwidefhom}
	\widetilde{f}_\hom(Z)\leq\limsup_{T\to+\infty}\inf\biggl\{ \underset{x\in TY}{\mbox{\rm ess-sup}}\hspace{0.1cm} f_\infty(x, Z + Du(x))\hspace{0.03cm}:\hspace{0.03cm} u\in W^{1,\infty}_{\#} (TY;\hspace{0.03cm}\R^d) \biggr \}.
	\end{equation}
Indeed, for fixed $T$ and $Z\in\R^{d\times n}$, we deduce that
     \begin{align}
         {1\over T^n} &\inf\left\{ \int_{TY} \widetilde{f^p}(x, Z+Du(x))dx \hspace{0.03cm}:\hspace{0.03cm} u\in W^{1,p}_{\#}(TY; \hspace{0.03cm} \R^d) \right\}\notag\\
         & \le \, \inf\left\{  \underset{x\in TY}{\mbox{\rm ess-sup}}\hspace{0.1cm} \widetilde{f^p}(x, Z + Du(x))\hspace{0.03cm}:\hspace{0.03cm} u\in W^{1,\infty}_{\#} (TY;\hspace{0.03cm}\R^d)     \right\}.\notag
     \end{align}
Taking the limit as $T\to+\infty$, it yields to 
   \begin{align}
       \notag
       f^\hom_p(Z) \leq \limsup_{T\to+\infty} \inf\left\{  \underset{x\in TY}{\mbox{\rm ess-sup}}\hspace{0.1cm} \widetilde{f^p}(x, Z + Du(x))\hspace{0.03cm}:\hspace{0.03cm} u\in W^{1,\infty}_{\#} (TY;\hspace{0.03cm}\R^d)     \right\},
   \end{align}
which implies that 
     \begin{align}
         \notag
         (f^\hom_p)^{1/p} (Z) \leq \limsup_{T\to+\infty} \inf\left\{  \underset{x\in TY}{\mbox{\rm ess-sup}}\hspace{0.1cm} (\widetilde{f^p})^{1/p}(x, Z + Du(x))\hspace{0.03cm}:\hspace{0.03cm} u\in W^{1,\infty}_{\#} (TY;\hspace{0.03cm}\R^d)     \right\},
     \end{align}
If $p \equiv k$ is an integer number, thanks to \eqref{def:finfinity}, it is easy to show that 
    \begin{align}
         \notag
         (f^\hom_k)^{1/k} (Z) \leq \limsup_{T\to+\infty} \inf\left\{  \underset{x\in TY}{\mbox{\rm ess-sup}}\hspace{0.1cm} f_\infty(x, Z + Du(x))\hspace{0.03cm}:\hspace{0.03cm} u\in W^{1,\infty}_{\#} (TY;\hspace{0.03cm}\R^d)     \right\},
     \end{align}    
and taking the limit as $k\to+\infty$ we conclude that \eqref{upperboundwidefhom} holds.
Now, consider a divergent sequence $\{p_k\}_k$ as $k\to+\infty$. By similar arguments as those in the proof of Lemma \ref{lemma:recoveryseq}, we obtain that 
    \begin{align}
        (f^\hom_{p_k})^{1/p_k} (Z) \leq \limsup_{T\to+\infty} \inf\left\{  \underset{x\in TY}{\mbox{\rm ess-sup}}\hspace{0.1cm} f_\infty(x, Z + Du(x))\hspace{0.03cm}:\hspace{0.03cm} u\in W^{1,\infty}_{\#} (TY;\hspace{0.03cm}\R^d)     \right\}.\notag
    \end{align}
Taking the limit as $p_k\to+\infty$ yields to  \eqref{upperboundwidefhom}.	
\end{remark}  

\begin{lemma}
\label{lemma:convergenceofhomdensity}
Let $f:\R^n\times\R^{d\times n}\to [0,+\infty]$ be a Borel function, $1$-periodic in the first variable satisfying the growth conditions \eqref{growthconditions}. Then, for all open bounded set $\Omega\subset \R^n$ and $u\in W^{1,\infty} (\Omega;\hspace{0.03cm}\R^d)$, we have that
        \begin{equation}
        \notag
        \lim_{p\to+\infty}\left(\int_{\Omega} f^\hom_p(Du(x))dx \right)^{1/p}=\underset{x\in\Omega}{\mbox{\rm ess-sup}}\hspace{0.1cm} \widetilde{f}_\hom(Du(x)),
        \end{equation}
where $\widetilde{f}_\hom$ is given by \eqref{characterizationftildehom}.
\end{lemma} 
\begin{proof}
First, note that since $(f^\hom_p)^{1/p}$ is a non-decreasing function, we have that, for all $Z\in\R^{d\times n}$,
    \begin{equation}
    \notag
    (f^\hom_p(Z))^{1/p}\leq \widetilde{f}_\hom(Z) \qquad\mbox{for any } p>1.
    \end{equation}
This implies that, for fixed $u\in W^{1, \infty}(\Omega;\hspace{0.03cm} \R^d)$, 
     \begin{equation}
     \notag
     \left[\int_{\Omega} f^\hom_p(Du(x))dx\right]^{1/p}\leq\left[\int_{\Omega} (\widetilde{f}_\hom(Du(x)))^pdx\right]^{1/p}= \|\widetilde{f}_\hom(Du(\cdot))\|_{L^p(\Omega)}. 
     \end{equation}
Taking the limit as $p\to+\infty$, we have that 
    \begin{equation}
    \notag
    \lim_{p\to+\infty}\left[\int_{\Omega} f^\hom_p(Du(x))dx\right]^{1/p} \leq 
    \|\widetilde{f}_\hom(Du(\cdot))\|_{L^\infty(\Omega)} =\underset{x\in\Omega}{\mbox{ ess-sup}}\hspace{0.03cm} \widetilde{f}_\hom (Du(x)).
    \end{equation}
Now, we prove the reverse inequality. To that end, fix $\e>0$ and $u\in W^{1, \infty}(\Omega;\hspace{0.03cm}\R^d)$. We define the set $E_\e$ by
   \begin{equation}
   \notag
   E_\e := \left\{x\in\Omega\hspace{0.03cm}:\hspace{0.03cm} \widetilde{f}_\hom(Du(x))> \esssup \widetilde{f}_\hom(Du(x))-{\e\over 2}  \right\}.
   \end{equation}
Note that $E_\e$ has positive measure, i.e. $m_\e:= {\mathcal L}^n(E_\e)>0$.  We consider $\delta$ such that $0<\delta<<m_\e$. 
From \eqref{limfhomp} combined with the Egorov theorem, it follows that there exists a set $F_\delta\subset\Omega$ such that ${\mathcal L}^n(F_\delta)\leq\delta$ and 
     \begin{equation}
     \notag
     \lim_{p\to+\infty} \|(f^\hom_p)^{1/p}(Du(\cdot)) -\widetilde{f}_\hom (Du(\cdot)) \|_{L^\infty(\Omega\setminus F_\delta)}=0.
     \end{equation}
In particular, there exists $p_0=p_0(\e)$ such  $p\geq p_0$ implies that
   \begin{equation}
   (f^\hom_p)^{\color{blue}{1/p}}(Du(x))  - \widetilde{f}_\hom (Du(x))\geq -{\e\over 2}, \qquad \forall x\in \Omega\setminus F_\delta.\notag
   \end{equation}
Note that ${\mathcal L}^n(E_\e\setminus F_\delta)>0$, since $0<\delta<<m_\e$. Moreover, if $x\in E_\e\setminus F_\delta$, then 
    \begin{align}
    (f^\hom_p)^{1\over p} (Du(x))&\geq \widetilde{f}_\hom (Du(x))-{\e\over 2}\notag\\
    &\geq \esssup \hspace{0.1cm} \widetilde{f}_\hom (Du (x)) -\e, \qquad\forall p\geq p_0.\notag
    \end{align}
Let $E_\e^p$ be the set defined by
    \begin{equation}
    \notag
    E_\e^p:=\left\{x\in\Omega\hspace{0.03cm}:\hspace{0.03cm} (f^\hom_p)^{1/p} (Du(x))> \esssup \hspace{0.1cm}\widetilde{f}_\hom (Du (x)) -\e   \right\}.
    \end{equation}
The set $E_\e^p$ contains $E_\e\setminus F_\delta$ which implies that the measure of $E_\e^p$ is positive for any $p\geq p_0$. Hence, we deduce that 
    \begin{align}
    \left(\int_{\Omega} f^\hom_p(Du(x))dx \right)^{1/p} &\geq \left(\int_{E_\e^p} f^\hom_p(Du(x))dx \right)^{1/p}\notag\\
    &\geq\left(\int_{E_\e^p} (\esssup \hspace{0.1cm}\widetilde{f}_\hom (Du(x)) -\e )^pdx \right)^{1/p}\notag\\
    &= {\mathcal L}^n (E^p_\e)^{1/p} \left(\esssup\hspace{0.1cm} \widetilde{f}_\hom (Du(x)) -\e \right)\notag\\
    &\geq {\mathcal L}^n (E_\e\setminus F_\delta)^{1/p} \left(\esssup \hspace{0.1cm}\widetilde{f}_\hom (Du(x)) -\e \right)\notag,
    \end{align} 
for any $p\geq p_0$. This implies that 
    \begin{align}
    \lim_{p\to+\infty} \left(\int_{\Omega} f^\hom_p(Du(x))dx \right)^{1/p} \geq \esssup \hspace{0.1cm}\widetilde{f}_\hom (Du(x)) -\e .\notag
    \end{align}
In view of the arbitrariness of $\e$, we obtain the reverse inequality, as desired.
\end{proof}

In the next proposition, we show that the function $\widetilde{f}_\hom$ \color{blue} is \color{black} strong Morrey quasiconvex (see Definition \ref{sMqcx}).

\begin{prop}
The function $\widetilde{f}_\hom$ is strong Morrey quasiconvex.
\end{prop}
\begin{proof}
First, note that $f^\hom_p: \R^{d\times n}\to [0,+\infty)$ is quasiconvex (or Morrey convex, according to \cite[Definition 1.1]{BJW04}) since it is the energy density of $\Gamma$-limit of $F_p^\hom$ (e.g., see \cite[Theorem 14.5]{BD}). In view of \cite[Proposition 2.4]{BJW04}, $f^\hom_p$ is also strong Morrey quasiconvex. 
Now, we show that for any $p>1$, $(f^\hom_p)^{1/p}$ is strong Morrey quasiconvex. To that end, fix $\e>0$ and $k>0$. Since $f^\hom_p$ is strong Morrey quasiconvex, there exists $\delta=\delta(\e, k)>0$ such that if $\varphi\in W^{1,\infty}(\Omega; \hspace{0.03cm}\R^d)$ satisfies $\|D\varphi\|_{L^\infty(Y; \hspace{0.03cm} \R^{d\times n})}\leq k$ and $\max_{x\in\partial Q}|\varphi(x)|\leq \delta$, then
    \begin{equation}
    \label{ineq2}
    f^\hom_p(Z) \leq \underset{x\in Y}{\mbox{ess-sup}}\hspace{0.1cm}f^\hom_p(Z + D\varphi(x)) + \e^p.
    \end{equation}
Recall that for any $p>1$, the function $h(x):= x^{1/p}$ is subadditive since it is concave function  with $h(0)\geq 0$.  Hence, from  \eqref{ineq2}, it follows that
   \begin{align}
   (f^\hom_p(Z))^{1/p}&\leq \left(\underset{x\in Y}{\mbox{ess-sup}}\hspace{0.1cm}f^\hom_p(Z + D\varphi(x)) + \e^p \right)^{1/p}\notag\\
   &\leq \left(\underset{x\in Y}{\mbox{ess-sup}}\hspace{0.1cm}f^\hom_p(Z + D\varphi(x))\right)^{1/ p} + \e\notag\\
   &= \underset{x\in Y}{\mbox{ess-sup}}\hspace{0.1cm}(f^\hom_p)^{1/p}(Z + D\varphi(x)) + \e.\label{ineq3}
   \end{align}

Finally, inequality \eqref{ineq3} implies that $(f^\hom_p)^{1/p}$ is strong Morrey quasiconvex.\par 
To conclude the proof, recall that $(f^\hom_p)^{1/p}$ is a non-decreasing function. Hence, for any $Z\in\R^{d\times n}$,
   \begin{equation}
   \notag
   \widetilde{f}_\hom (Z) = \lim_{p\to+\infty}  (f^\hom_p(Z))^{1/p} = \sup_{p\geq 1} (f^\hom_p(Z))^{1/p}.
   \end{equation}
Moreover, due to the fact that for any $p\geq 1$, $(f^\hom_p(Z))^{1/p}$ is strong Morrey quasiconvex, we may apply \cite[Proposition 5.2]{PZ20} to conclude that $\widetilde{f}_\hom$ is strong Morrey quasiconvex, as desired.
\end{proof}

The next result provides the second $\Gamma$-limit as $p\to+\infty$ of the functionals $\{F_{p, \e}\}_{p,\e}$ in \eqref{def:Functionalpeps}, after having taken the one with respect to $\varepsilon \to 0$, cf. \eqref{GammaLp} and \eqref{def:Fhomp}, and it is   a generalization of \cite[Theorem 3.3]{BGP04} in the vectorial case.

\begin{thm}

\label{Theorem3.6}

Let $f:\R^n\times \R^{d\times n}\to [0,+\infty)$ be a Borel function, $1$-periodic in the first variable and satisfying growth conditions \eqref{growthconditions}. Let $\Omega$ be a bounded, open set of $\R^n$ with Lipschitz boundary. For any $p>1$, let $F^\hom_p$ be the functional defined by \eqref{def:Fhomp}, and $\widetilde{f}_\hom$ be the function in \eqref{characterizationftildehom}. Then,
   \begin{equation}
   \notag   \Gamma(L^\infty)\mbox{-}\lim_{p\to\infty} F_p^{\rm hom}(u)= F^\hom(u):=  \underset{x\in \Omega}{\mbox{{\rm ess-sup}}}\hspace{0.1cm} \widetilde{f}_\hom (Du(x)),
   \end{equation}for every $u \in W^{1, \infty}(\Omega; \mathbb{R}^d)$.
\end{thm}    
\begin{proof}
For any $p>1$, by \eqref{GammaLp} and \cite[Proposition 6.16]{DM93}, the functional 
   \begin{equation*}
   H^\hom_p(u)=\left({1\over {\mathcal L^n(\Omega) }}\int_{\Omega} f^\hom_p(Du(x)) dx\right)^{1/p}
   \end{equation*}
is lower semicontinuous in $W^{1,p}$ with respect to $L^p$ topology, hence it \color{blue}results \color{black} to be lower semicontinuous in $W^{1, \infty}$ with respect to the $L^\infty$ topology. For the sake of exposition, we assume in the rest of the proof that $\mathcal L^n(\Omega)=1$. Moreover, $H^\hom_p$ is an increasing family, i.e., for $p_1\leq p_2$, $H^\hom_{p_1}\leq  H^\hom_{p_2}$. Indeed, using the Jensen inequality combined with the fact that $(f^\hom_p)^{\color{blue}{1/p}}$ is an increasing sequence {\color{blue} with respect to $p>1$}, we deduce that, for $p_1\leq p_2$,
    \begin{align}
    \left(\int_{\Omega}f^\hom_{p_1}(Du(x))dx\right)^{p_2/p_1}&\leq
    \int_{\Omega}(f^\hom_{p_1}(Du(x)))^{p_2/p_1}dx\notag\\
    &\leq \int_{\Omega}f^\hom_{p_2}(Du(x))dx\notag,
    \end{align}  
which implies that $H^\hom_{p_1}(u)\leq H^\hom_{p_2}(u)$. 
\par In view of Lemma \ref{lemma:convergenceofhomdensity}, it follows that, for any $u\in W^{1,\infty}(\Omega;\hspace{0.03cm}\R^d)$,
    \begin{equation}
    \notag
    \lim_{p\to+\infty}H^\hom_p(u)=\underset{x\in \Omega}{\mbox{ess-sup}}\hspace{0.03cm} \widetilde{f}_\hom (Du(x)). 
    \end{equation}
In other words, $H^\hom_p$ converges pointwise to $F^\hom$, as $p\to+\infty$. This combined with the lower semicontinuity of $H^\hom_p$ allows us to apply \cite[Remark 5.5]{DM93}, concluding that  $H^\hom_p$ $\Gamma$-converges to $F^\hom$ with respect to $L^\infty$ topology.
\par To conclude the proof, it remains to prove that $F^\hom_p$ $\Gamma$-converges to  $F^\hom$ with respect to $L^\infty$ topology. This is a consequence of $\Gamma$-convergence of $H^\hom_p$ combined with the equality 
    \begin{equation}
    \notag
    \lim_{p\to+\infty} H^\hom_p(u)= \lim_{p\to+\infty} F^\hom_p(u).
    \end{equation}
This concludes the proof.
\end{proof}

\begin{remark}\rm
The previous result along with the homogenization result,  given by \eqref{GammaLp} and \eqref{def:Fhomp},  for the functionals $\{F_{p,\e}\}_{p,\e}$ in \eqref{def:Functionalpeps}  yields to 
    \begin{equation}
    \notag
   \Gamma(L^\infty)\mbox{-}\lim_{p\to +\infty}\left(\Gamma(L^p)\mbox{-}\lim_{\e\to 0} F_{p.\e}(u)\right)=\underset{x\in \Omega}{\mbox{ess-sup}}\hspace{0.03cm} \widetilde{f}_\hom (Du(x)),
    \end{equation}
    for every $u \in W^{1, \infty}(\Omega; \mathbb{R}^d)$.
\end{remark}

\section{The cell formula}
\color{black}

\label{quattro}

The effective energy density  $\widetilde{f}_\hom$ in \eqref{characterizationftildehom} is characterized through an asymptotic formula. In this section, we show that if the function $f(x,\cdot)$ is level convex and upper semicontinuous for every $x\in \R^n$, {\color{blue}the} asymptotic formula turns into a cell formula. First, we recall the definition of level convexity.

  \begin{definition}
     Let $\Omega \subseteq \mathbb R^n$ be an open set. A 
     function $f:\R^{d\times n} \to \R$ is level convex if for any $t\in (0,1)$ and for any $Z_1, Z_2\in \R^{d\times n}$, it holds
          \begin{equation}
              \notag
              f(t Z_1+ (1-t)Z_2)\leq \color{blue}{\rm max}\{f( Z_1), f(Z_2)\}\color{black}.
          \end{equation}
  \end{definition}
  The following result will be useful in the rest of the paper.
  
\begin{prop}\label{lcus}
 Let $\Omega \subseteq \mathbb R^n$ be an open set.
Let $f:\Omega\times \R^{d\times n}\to [0,+\infty)$ be a $\mathcal L(\mathbb R^n)\otimes \mathcal B(\mathbb R^{d\times n})$-measurable function, 
satisfying growth conditions \eqref{growthconditions}. Assume that $f(x, \cdot)$ is level convex and upper semicontinuous for a.e. $x\in\Omega$. Then, $f_\infty$ defined by \eqref{def:finfinity} \color{blue} is \color{black} level convex.
\end{prop}
\begin{proof}
The proof follows from \cite[Remark 5.2]{PZ20}. Indeed, we have that 
    \begin{equation}
        \label{finftyQinfty}
        f_\infty(x, \cdot) = Q_\infty f (x, \cdot) = f^{\rm ls}(x,\cdot) \quad \mbox{for a.e. } x\in\Omega,
    \end{equation}
and $f^{\rm ls}(x, \cdot)$ is level convex, see \cite[(i), Proposition 2.3]{RZ}.
\end{proof}

\color{black}
\begin{remark}\rm
    \label{cexButtazzo}
     Note that, in the above proposition, it is not possible to drop the upper semicontinuity assumption in order to get the equality in \eqref{finftyQinfty}. Indeed arguing as in \cite[Example 4.4.6]{Buttazzo} without upper semicontinuity it is possible to show that the  sequentially weakly $W^{1,p}$ lower semicontinuous envelope of the functional $G(u):= \int_\Omega  f^p(x,D u) dx$, where $$f^p(x,Z):=\left\{\begin{array}{ll}
     |Z|^p & \hbox{ if } Z= (x_n,0,\dots, 0),\\
     &\\
     (1+|Z|)^p &\hbox{ otherwise},
     \end{array}
     \right.$$
     is expressed by $\int_\Omega \widetilde{f^p}(x,D u)dx$, where $\widetilde{f^p}(x,Z):= (1+|Z|)^p > (f^p)^{\ast \ast}(x,Z)= Q (f^p)(x,Z)=|Z|^p$, consequently for this same function $Q_\infty f(x,Z)= |Z|< f_\infty(x,Z)=1+ |Z|$.
\end{remark}

The next proposition provides sufficient conditions to ensure that the asymptotic formula \eqref{characterizationftildehom} turns into a single cell formula.

\begin{prop}
Let $f:\R^n\times \R^{d\times n}\to [0,+\infty)$ be a Borel function, $1$-periodic in the first variable satisfying growth conditions \eqref{growthconditions}. Assume that $f(x, \cdot)$ is level convex 
for every $x\in\R^n$. Then, asymptotic formula \eqref{characterizationftildehom} is reduced to the following cell problem
    \begin{equation}
    \label{cellformula}
        \widetilde{f}_\hom(Z) = \inf\left\{\underset{x\in Y}{\mbox{{\rm ess}-{\rm sup}}}\hspace{0.1cm} 
 f_\infty(x, Z+Du(x))\hspace{0.03cm}:\hspace{0.03cm} u\in W^{1,\infty}_{\#} (Y; \hspace{0.03cm} \R^d)\right\},
    \end{equation}
for all $Z\in \R^{d\times n}$.
\end{prop}
\begin{proof}
    For $j\in\mathbb{N}$, we have that $W^{1,\infty}_{\#} (Y; \hspace{0.03cm} \R^d)\subset W^{1,\infty}_{\#} (jY; \hspace{0.03cm} \R^d)$, which implies that
       \begin{align}
       \inf&\left\{\underset{x\in jY}{\mbox{{\rm ess}-{\rm sup}}}\hspace{0.1cm} 
 f_\infty(x, Z+Du(x))\hspace{0.03cm}:\hspace{0.03cm} u\in W^{1,\infty}_{\#} (jY; \hspace{0.03cm} \R^d)     \right\}\notag\\
 &\leq \inf\left\{\underset{x\in Y}{\mbox{{\rm ess}-{\rm sup}}}\hspace{0.1cm} 
 f_\infty(x, Z+Du(x))\hspace{0.03cm}:\hspace{0.03cm} u\in W^{1,\infty}_{\#} (Y; \hspace{0.03cm} \R^d)     \right\}.\notag
       \end{align}
Taking the infimum over $j\in\mathbb{N}$, we easily conclude that
   \begin{equation}
       \notag
       \widetilde{f}_\hom(Z) \leq \inf\left\{\underset{x\in Y}{\mbox{{\rm ess}-{\rm sup}}}\hspace{0.1cm} 
 f_\infty(x, Z+Du(x))\hspace{0.03cm}:\hspace{0.03cm} u\in W^{1,\infty}_{\#} (Y; \hspace{0.03cm} \R^d)     \right\}.
   \end{equation}
On the other hand, for $v\in W^{1, \infty}_{\#}(jY;\hspace{0.03cm} \R^d)$, we set
    \begin{equation}\notag
        u{\color{blue}(x)}:= \sum_{i\in I} {1\over j^n} v(x+i),
    \end{equation}
with $I=\{0, 1, \dots, j-1\}^n$. One can easily prove that $u\in W^{1, \infty}_{\#}(Y;\hspace{0.03cm} \R^d)$. Using the level convexity and the periodicity of $f_\infty$, we deduce that 
    \begin{align}
        \underset{x\in Y}{\mbox{{\rm ess}-{\rm sup}}}\hspace{0.1cm}f_\infty(x, Z+Du(x)) &=
       \underset{x\in Y}{\mbox{{\rm ess}-{\rm sup}}}\hspace{0.1cm}f_\infty(x, Z+\sum_{i\in I} {1\over j^n} Dv(x+i)) \notag\\
       &\leq \underset{x\in Y}{\mbox{{\rm ess}-{\rm sup}}}\hspace{0.1cm}\max_{i\in I} f_\infty(x, Z + Dv(x+i))\notag
       \end{align}
       \begin{align}
       &\leq \max_{i\in I} \underset{x\in Y}{\mbox{{\rm ess}-{\rm sup}}}\hspace{0.1cm} f_\infty(x, Z + Dv(x+i))\notag\\
        &=\max_{i\in I} \underset{y\in (i,1+i)^n}{\mbox{{\rm ess}-{\rm sup}}}\hspace{0.1cm} f_\infty(y-i, Z + Dv(y))\notag\\
        &=\max_{i\in I} \underset{y\in (i,1+i)^n}{\mbox{{\rm ess}-{\rm sup}}}\hspace{0.1cm} f_\infty(y, Z + Dv(y))\notag\\
        &\leq \underset{y\in jY}{\mbox{{\rm ess}-{\rm sup}}}\hspace{0.1cm} f_\infty(y, Z + Dv(y)),\notag
    \end{align}
where we have performed the change of variables $y=x+i$. Now, taking the infimum yields to 
    \begin{align}
        \inf &\left\{ \underset{x\in Y}{\mbox{{\rm ess}-{\rm sup}}}\hspace{0.1cm}f_\infty(x, Z+Du(x))\hspace{0.03cm}:\hspace{0.03cm} u\in W^{1, \infty}_{\#}(Y; \hspace{0.03cm} \R^d) \right\}\notag\\
        &\leq \inf \left\{ \underset{x\in jY}{\mbox{{\rm ess}-{\rm sup}}}\hspace{0.1cm}f_\infty(x, Z+Du(x))\hspace{0.03cm}:\hspace{0.03cm} u\in W^{1, \infty}_{\#}(jY; \hspace{0.03cm} \R^d) \right\}.\notag
    \end{align}
Finally, taking the infimum over $j\in\mathbb{N}$, we conclude that 
    \begin{align}
        \inf \left\{ \underset{x\in Y}{\mbox{{\rm ess}-{\rm sup}}}\hspace{0.1cm}f_\infty(x, Z+Du(x))\hspace{0.03cm}:\hspace{0.03cm} u\in W^{1, \infty}_{\#}(Y; \hspace{0.03cm} \R^d) \right\}\leq \widetilde{f}_\hom(Z),\notag
    \end{align}
as desired.
\end{proof}
\begin{remark} \rm \label{rmk:cellformula}
If  $f(x, \cdot)$ is level convex and upper semicontinuous for every $x\in\R^n$, in view of Proposition \ref{lcus}, the cell formula \eqref{cellformula} may be specialized as follows
       \begin{align}
           \widetilde{f}_\hom(Z) &=\inf \left\{ \underset{x\in Y}{\mbox{{\rm ess}-{\rm sup}}}\hspace{0.1cm} Q_{\infty} f(x, Z+Du(x))\hspace{0.03cm}:\hspace{0.03cm} u\in W^{1, \infty}_{\#}(Y; \hspace{0.03cm} \R^d) \right\}\notag\\
           &=\inf \left\{ \underset{x\in Y}{\mbox{{\rm ess}-{\rm sup}}}\hspace{0.1cm}f^{\rm ls}(x, Z+Du(x))\hspace{0.03cm}:\hspace{0.03cm} u\in W^{1, \infty}_{\#}(Y; \hspace{0.03cm} \R^d) \right\}.\notag
       \end{align}
\end{remark}

{
\section{Homogenization of unbounded functionals}
\label{secHomunb}

In this section, we provide a homogenization result via $\Gamma$-convergence of the family of unbounded integral functionals of the form
    \begin{equation}
        \label{functCDA}
        G_\e(u) := 
        \begin{dcases}
            \int_{\Omega} g\left({x\over \e}, Du(x)\right)dx & \quad \mbox{for } u\in W^{1,q}_{\loc} (\R^n; \hspace{0.03cm} \R^d)\cap L^\infty_{\loc} (\R^n; \hspace{0.03cm} \R^d), \\
            &\\
            +\infty & \quad \mbox{otherwise,}
        \end{dcases}        
    \end{equation}
    
where $\Omega$ is a bounded, open, convex set, $q\in [1, +\infty]$ and the energy density $g: \R^n\times \R^{d \times n} \to [0, +\infty]$ satisfies the following assumptions:
\begin{itemize}
    \item [(H1)] $g$ is ${\mathcal L}(\R^n)\otimes {\mathcal B}(\R^{d \times n})$-measurable,
    \item[(H2)] $g$ is  $Y$-periodic in the first variable and convex in the second one.
\end{itemize} 
{\color{blue} We point out that an open convex set $\Omega$ of $\R^n$ has Lipschitz boundary (see, e.g., \cite[Proposition 2.5.1]{CDA02}).}

For any $q\in[1, +\infty]$, we introduce the homogenized energy density $\widetilde{g}^q_\hom: \R^{d\times n}\to \R$ defined by
     \begin{equation}
     \label{ftildeCDA}
     \widetilde{g}^q_\hom(Z) \coloneqq \inf \left \{ \int_Y g(y, Z + Dv) \, dy: v \in W^{1,q}_{\#} (Y; \mathbb{R}^d) \cap L^{\infty}(Y; \R^d)\right \}.
     \end{equation}
Throughout this section, we also require the following technical assumptions:
     \begin{itemize}
         \item[(H3)] there exists $\delta\in(0,1)$ such that 
   \begin{equation}
       \label{existdelta}
      B_{2\delta}(0)\subseteq {\rm int}({\rm dom}\widetilde{g}^q_\hom).
        \end{equation}
      \item[(H4)]
       We assume that there exist matrices $\{A_i\}_{i=1}^{nd}\subset \R^{d\times n}$, vertices of the cube $Q$ such that ${\rm B}_\delta(0) \subset Q \subset {\rm B}_{2\delta}(0)$ and, for any $i=1, \cdots, nd$,     \begin{equation}
           \label{hpL}
           \int_Y g(y, A_i)dy< +\infty.
       \end{equation}

     \end{itemize}
\begin{remark}\rm
\label{remark:existenceofcube}
     As in {\rm \cite{CDA02}}, assumption {\rm (H3)}, in the scalar case, allows us to find a cube $Q$ whose vertices $\{A_i\}_{i=1}^{nd}$ are contained in the ball $B_{2\delta}(0)$ in the effective domain of $g(x,\cdot)$ for a.e. $x \in \Omega$. 
     This fact is not evident in the vectorial setting. This is the reason why we introduced the technical assumption ${\rm (H4)}$.

       In turn, ${\rm (H4)}$, together with ${\rm (H3)}$ and {\rm \cite[Proposition 1.1.15]{CDA02}}, due to the convexity of $\widetilde g^q_{\rm hom}$, guarantees that
       $$ \widetilde g_{\rm hom}^q(A_i)<+ \infty,$$
       for the same $A_i$ appearing in \eqref{hpL}.
       
Analogously, the convexity of $g$ and \eqref{hpL} entail that \begin{equation}\label{H40}
g(\cdot, 0)\in L^1_{\rm loc}(\R^n),
\end{equation} 
where $0$ denotes the null matrix in $\R^{d\times n}$. 

The same observations made at the beginning, allow us to prove the $\Gamma (L^\infty)$-limit of our sequence \eqref{functCDA}, dealing with sequences  $\{\e_h\}_h$, converging to $0^+$ as $h\to +\infty$ extracted by the vanishing family $\{\varepsilon\}$. Moreover it is sufficient to replace the generic sequence  $\{\e_h\}_h$ by $\{\frac{1}{h}\}$, since, due to hypotheses {\color{blue} {\rm (H1)}-{\rm (H4)}}, we can argue exactly as in \cite[Lemma 12.1.2]{CDA02}.
\color{black}

 \color{blue} We also observe that on the one hand the convexity assumption {\rm (H2)} might appear quite strong compared with the more natural quasiconvexity notion, since we are in the vectorial setting. On the other hand, we are dealing with unbounded functionals, for which quasiconvexity is not completely understood, as it is not clear what is the quasiconvex hull of sets of generic matrices.

 Furthermore, we emphasize that the class of functions $g$ which satisfies {\rm (H1)}-{\rm (H4)}, is clearly not empty. 
 As a first example we could consider $g(x,Z):= a(x){\bf 1}_Q(Z)$, where $Q$ represents the cube in $\mathbb R^{d\times n}$ with vertices $\{A_i\}_{i=1}^{nd}$, and $a(x)$ represents any periodic function in $L^\infty$, such that $a(x)>C$, for a suitable $C>0$.
 Such a function $g$ clearly satisfies all the assumptions {\rm (H1)}-{\rm (H4)}. In particular, regarding {\rm (H3)}, we observe that $\tilde g^q_{hom}(Z)\geq C I_Q (Z)$.

 As a second example, we consider the indicator function of any level set of the function $f(x,Z)= a(x)\Psi(Z)$, where $\Psi$ is given by \eqref{Psidef}. Indeed this latter function belongs to the class of functions satisfying the assumptions in Theorem \ref{Theorem5.1}, for which the indicator functions of the level sets fit in the sets of hypotheses of Theorem \ref{Proposition 12.6.1}.
 \color{black}
\end{remark}

Next, we state some properties of the homogenized energy density $\widetilde{g}_{\hom}^q$ given by \eqref{ftildeCDA}, that will be used in the sequel. The proof of such properties follows the same arguments of \cite[Proposition 12.1.3]{CDA02} and for this reason, it is omitted.

\begin{prop}
    \label{prop:Proposition 12.1.3} 
    Let $g$ be the energy density satisfying {\rm (H1)} and  {\rm (H2)}. Let $q\in[1,+\infty]$ and let $\widetilde{g}_{\hom}^q$ be given by \eqref{ftildeCDA}. 
       \begin{itemize}
           \item [{\rm (i)}] It \color{blue} results \color{black} that $\widetilde{g}^q_{\hom}$ is convex.
           \item[{\rm (ii)}] Assume that  $p\in [1,+\infty]$ and $q\in[p, +\infty]$. In addition, assume that $g$ is such that
               \begin{equation}
               \label{crescitagtildehom}
                   \begin{dcases}
                       |Z|^p\leq g(x, Z) & \mbox{for a.e. } x\in\R^n \hspace{0.1cm} \mbox{for any } Z\in\R^{d\times  n} \hspace{0.1cm} \mbox{if } p\in[1, +\infty),\\
                       {\rm dom}g(x, \cdot)\subseteq B_R(0) & \mbox{for a.e. } x\in\R^n \hspace{0.1cm} \mbox{if } p=+\infty.
                   \end{dcases}
               \end{equation}
              Then, $\widetilde{g}^q_{\hom}$ satisfies 
              \begin{equation}
               \notag
                   \begin{dcases}
                       |Z|^p\leq \widetilde{g}^q_{\hom}(x, Z) &  \mbox{for any } Z\in\R^{d\times  n} \hspace{0.1cm} \mbox{if } p\in[1, +\infty),\\
                       {\rm dom}\widetilde{g}^q_{\hom}\subseteq B_R(0) & \mbox{if } p=+\infty.
                   \end{dcases}
               \end{equation}
                 \item[{\rm (iii)}] Let  $q\in (n,+\infty]$. Then,
             \begin{equation*}
                 \widetilde{g}^q_{\hom}(Z)= \inf\left\{ \int_Y g(y, Z+Dv)dy : v\in W^{1, q}_{\#}(Y; \R^d) \right\}.
             \end{equation*}
        \item[{\rm (iv)}] Assume that $p\in(n, +\infty]$ and $p=q$ {\color{blue} in \eqref{crescitagtildehom}}. Furthermore, assume that $g(x, \cdot)$ is lower semicontinuous for a.e. $x\in\R^n$. Then, $\widetilde{g}^q_{\hom}$ is lower semicontinuous and 
            \begin{equation}
                \notag
                \widetilde{g}^q_{\hom}(Z) = \min\left\{\int_Y g(y, Z+Dv)dy : v\in W^{1, q}_{\#}(Y; \R^d) \right\},
            \end{equation}
        for any $Z\in\R^{d\times n}$.
      
       \end{itemize}
\end{prop}

The main result of this section is the following.
\begin{thm}
\label{Proposition 12.6.1}
    Let $g:\mathbb R^n \times \mathbb R^{d\times n}\to [0, +\infty]$ be satisfying {\rm (H1)}-{\rm (H4)}.  
    Then, the functionals $\{G_\e\}_{\e}$ given by \eqref{functCDA} $\Gamma(L^\infty)$-converge to the functional $G_{\hom}$ defined by 
          \begin{equation}
              \notag
              G_{\hom}(\Omega, u)\coloneqq \int_{\Omega} (\widetilde{g}^q_{\hom}(Du))^{\rm ls}dx,
          \end{equation}
    for any $\Omega\in{\cal T}_0$ and $u\in \bigcup_{s>n} W^{1,s}_{\loc}(\R^n; \R^d)$ and with $\widetilde{g}^q_{\hom}$ being defined by \eqref{ftildeCDA}.
\end{thm}
The proof is quite long and technical and it is provided at the end of this section as a consequence of all the preliminary results proved in the sequel. Indeed, we adapt the methodology used in \cite[Chapter 12]{CDA02}. Briefly,  we investigate some measure-theoretic properties of the $\Gamma$-limits. In particular, we prove that the $\Gamma$-liminf is super-additive and the $\Gamma$-limsup is sub-additive. The proof slightly differs from  \cite[Proposition 12.2.1]{CDA02} which uses arguments being appropriate only in the scalar setting. Indeed, the proof of these properties in the scalar setting relies on the existence of a particular family of cut-off functions (see \cite[Lemma 11.1.1]{CDA02}). The explicit construction of such cut-off functions may not be repeated in the vectorial case and hence we need to use other techniques to prove the measure theoretic aspect of $\Gamma$-limit (see Proposition \ref{Proposition 12.2.1}).  Then, we provide an integral representation on linear functions. Once again, the proof of such a result is different from \cite[Proposition 12.4.6]{CDA02}. Finally, we give a full representation result.  
}

\subsection{Measure representation results of $\Gamma$-limits}
\color{black}

In this subsection, we prove some measure theoretical aspects of the functionals $\{G_\e\}_{\e}$ defined by \eqref{functCDA}. To that end, for $h\in\mathbb{N}$, we consider a subsequence $G_{h}:= G_{\e_h}$ of the family $\{G_\e\}_{\e},$ where $\{\e_h\}_h$ is a positive decreasing sequence that goes to $0^+$ as $h \rightarrow + \infty$.
\\
{\color{blue} We recall the definition of increasing set functions, superadditvity on disjoint sets and subadditivity (see \cite[Definitions 2.6.1 and 2.6.5]{CDA02}).
\begin{definition}
Let $\mathcal{O}$ be the family of open subsets of $\mathbb R^n$  and $\alpha: \mathcal{O} \rightarrow [0, + \infty].$ We say that $\alpha$ is:
\begin{itemize}
    \item[{\rm (i)}] {\sc increasing set function} if
        \begin{equation*}
            \alpha(\Omega_1)\leq \alpha(\Omega_2),  
        \end{equation*}
    for every $\Omega_1, \Omega_2\in \mathcal{O}$ such that $\Omega_1\subseteq\Omega_2$;
    \item[{\rm (ii)}] {\sc superadditive} if 
    \begin{equation*}
        \alpha (\Omega_1) + \alpha(\Omega_2)  \le \, \alpha(\Omega),
    \end{equation*}
for every $\Omega_1, \Omega_2, \Omega \in \mathcal{O}$ with $\Omega_1 \cap \Omega_2 = \emptyset,$ $\Omega_1 \cup \Omega_2 \subseteq \Omega,$
    \item[{\rm (iii)}] {\sc subadditive} if
       \begin{equation*}
           \alpha(\Omega) \le \alpha (\Omega_1) + \alpha(\Omega_2), 
       \end{equation*}
for every $\Omega_1, \Omega_2, \Omega \in \mathcal{O}$ with $\Omega_1 \cap \Omega_2 = \emptyset,$ $\Omega \subseteq  \Omega_1 \cup \Omega_2 $.
\end{itemize}
\end{definition}

In what follows, we also make use of an inner regular envelope whose definition is recalled below (see \cite[Definition 15.5]{DM93})

\begin{definition}
    Let $G:\Omega\times {\cal T}(\Omega)\to [-\infty, +\infty]$ be an increasing functional. The inner regular envelope of $G$ is the increasing functional $G_{-}:\Omega\times {\cal T}(\Omega)\to [-\infty, +\infty]$ given by 
         \begin{align}
             \notag
             G_{-}(x, A) \coloneqq \sup\{G(x, B) : B\in {\cal T}(\Omega), B\subset\subset A \}. 
         \end{align}
\end{definition}
}

We  set
    \begin{eqnarray}
     && \widetilde{G}'(\Omega, \cdot): u \in L^{\infty}_{\rm loc}(\mathbb{R}^n; \mathbb{R}^d) \mapsto \Gamma(L^{\infty})\mbox{-}\liminf_{h \rightarrow + \infty} G_{\e_h}(\Omega, u), \label{defliminf}\\[2mm]
       && \widetilde{G}''(\Omega, \cdot): u \in L^{\infty}_{\rm loc}(\mathbb{R}^n; \mathbb{R}^d) \mapsto \Gamma(L^{\infty})\mbox{-}\limsup_{h \rightarrow + \infty} G_{\e_h}(\Omega, u).\label{deflimsup}
     \end{eqnarray}

The next result, whose proof is an adaptation of \cite[Proposition 12.2.1]{CDA02}, follows.
\begin{prop}
    \label{Proposition 12.2.1}
     Let $g: \R^n\times \R^{d \times n} \to [0, +\infty]$ satisfy  
     {\rm (H1)}-{\rm (H4)},
    and let $q\in[1, +\infty]$. 
     Let $\Omega, \Omega_1, \Omega_2\in {\cal T}_0$. 
           \begin{itemize}
               \item[{\rm (i)}]  If $\Omega_1\cap \Omega_2=\emptyset$ and $\Omega_1\cup\Omega_2\subset \Omega$, then 
       \begin{equation}
       \label{ine12.2.2}
           \widetilde{G}'_{-}(\Omega, u)\geq \widetilde{G}'_{-}(\Omega_1, u) + \widetilde{G}'_{-}(\Omega_2, u),
       \end{equation}
    for any $u\in L^\infty_\loc (\R^n; \R^d)$, where $\widetilde{G}'_{-}$ is the inner regular envelope of the functional in \eqref{defliminf}.
             \item[{\rm (ii)}]  If $\Omega\subset \Omega_1\cup \Omega_2$,  then 
       \begin{equation}
       \label{ine12.2.3}
           \widetilde{G}''_{-}(\Omega, u)\leq \widetilde{G}''_{-}(\Omega_1, u) + \widetilde{G}''_{-}(\Omega_2, u),
       \end{equation}
       for any $u\in L^\infty_\loc (\R^n; \R^d)$, with $\widetilde{G}''_{-}$ the inner regular envelope of the functional defined by \eqref{deflimsup}. 
           \end{itemize}
\end{prop}

\begin{proof}
    Inequality \eqref{ine12.2.2} follows from the definition of $\widetilde{G}'_{-}$. To show \eqref{ine12.2.3},  we note that, by \eqref{existdelta}, it holds 
\[
0 \in {\rm int} ({\rm dom} \widetilde{g}^q_\hom). 
\] 
   It is not restrictive to assume that $\Omega\subset\subset \Omega_1\cup\Omega_2$, so that we may simply prove that 
    \begin{equation}
    \label{ine12.2.4}
         \widetilde{G}''(\Omega, u)\leq \widetilde{G}''(\Omega_1, u) + \widetilde{G}''(\Omega_2, u),
    \end{equation}
    for any $u\in L^\infty(\R^n; \R^d)$. To that end, fix $u\in L^\infty(\R^n; \R^d)$ and assume that the right-hand side of \eqref{ine12.2.4} is finite. Therefore, for $i=1,2$, there exist two sequences $\{u_h^{i}\}\subset W^{1, q}_\loc(\R^n; \R^d)\cap L^\infty_\loc(\R^n; \R^d)$ such that  $u_h^i$ converges to $u$ in $L^\infty(\Omega_i; \R^d)$ and, for any positive decreasing sequence $\e_h \rightarrow 0^+$ as $h \rightarrow + \infty$
        \begin{equation}
        \label{eq2Prop12.2.1}
            \limsup_{h\to+\infty} \int_{\Omega_i}g \left(\frac{x}{\e_h}, Du^i_h \right)dx\leq \widetilde{G}''(\Omega_i, u).
        \end{equation}
    Since $\Omega\subset\subset \Omega_1\cup\Omega_2$, there exists $A_1\subset\subset\Omega_1$ such that $\Omega\subset\subset A_1\cup\Omega_2$. For every $h \in \mathbb N$, we construct $w_h\in W^{1,q}_\loc(\R^n; \R^d)\cap L^\infty_\loc(\R^n; \R^d)$ which is equal to $u^1_h$ in $\Omega_1$, and $w_h=u^2_h$ in $\Omega_2$. Let $\psi$ be a cut-off function such that $0\leq \psi\leq 1$ a.e. in $\Omega_1$, $\psi\equiv 1$ a.e. in $\overline{A_1}$, $\psi\equiv 0$ a.e. in $\R^n\setminus\Omega_1$, and such that 
     \begin{equation}
     \notag
         \|\nabla \psi\|_{L^\infty(\mathbb{R}^n)}  \leq {C\over \mbox{dist}(A_1, \partial \Omega_1)}
     \end{equation}
     holds for some constant $C>0.$ For any $h\in\N$, we define the vector-valued function $w_h$ by
       \begin{equation}
       \notag
           w_h\coloneqq \psi u^{1}_h + (1-\psi)u^{2}_h.
       \end{equation}
      Note that the sequence $\{w_h\}_h$ converges to $u$ in $L^\infty(\Omega; \R^d)$. For fixed $t\in[0,1)$, evoking the convexity of $g$ and recalling that $\Omega\setminus A_1\subset\subset \Omega_2$, it follows that 
      
         \begin{align}
             & \widetilde{G}''(\Omega, tu)\leq \limsup_{h\to +\infty} \int_\Omega g\left(\frac{x}{\e_h}, tDw_h\right )dx\notag\\
             &= \limsup_{h\to +\infty} \int_\Omega g\left(\frac{x}{\e_h}, t\psi Du^1_h + t(1-{\psi})Du^2_h + t(u^1_h-u^2_h)\otimes \nabla \psi \right)dx\notag\\
             &\leq \limsup_{h\to+\infty} t\int_\Omega g\left(\frac{x}{\e_h}, \psi Du^1_h + (1-\psi)Du^2_h\right)dx\notag\\
             &\quad + \limsup_{h\to+\infty} (1-t)\int_\Omega g\left(\frac{x}{\e_h}, {t\over 1-t}(u^1_h-u^2_h)\otimes \nabla \psi\right)dx\notag\\
             &\leq \limsup_{h\to+\infty} \int_\Omega \psi(x)g\left(\frac{x}{\e_h}, Du^1_h\right)dx + \limsup_{h\to +\infty} \int_\Omega (1-\psi(x))g\left(\frac{x}{\e_h}, Du^2_h\right)dx\notag\\
             &\quad + \limsup_{h\to +\infty} (1-t)\int_\Omega g\left(\frac{x}{\e_h}, {t\over 1-t}(u^1_h-u^2_h)\otimes \nabla \psi\right)dx\notag\\
             &\leq \limsup_{h\to+\infty} \int_{\Omega_1} g\left(\frac{x}{\e_h}, Du^1_h\right)dx + \limsup_{h\to+\infty} \int_{\Omega_2} g\left(\frac{x}{\e_h}, Du^2_h\right)dx\notag\\
             &\quad + \limsup_{h\to+\infty} (1-t)\int_\Omega g\left(\frac{x}{\e_h}, {t\over 1-t}(u^1_h-u^2_h)\otimes \nabla \psi\right)dx
             .\notag
         \end{align}

We observe  that $\nabla \psi$ is not identically equal to $0$ in $\Omega\cap(\Omega_1\setminus \overline{A_1})$ while $\nabla \psi=0$ in  $\Omega\setminus(\Omega_1\setminus \overline{A_1})$. This along with \eqref{eq2Prop12.2.1} implies that 
       \begin{align}
           \widetilde{G}''(\Omega, tu)& \leq \widetilde{G}''(\Omega_1, u) + \widetilde{G}''(\Omega_2, u)+(1-t)\limsup_{h\to+\infty} \int_{\Omega\setminus(\Omega_1\setminus \overline{A_1})} g\left(\frac{x}{\e_h}, 0\right)dx\notag\\
           &\quad + (1-t)\limsup_{h\to+\infty} \int_{\Omega\cap(\Omega_1\setminus \overline{A_1})} g\left(\frac{x}{\e_h}, {t\over 1-t}(u^1_h-u^2_h)\otimes \nabla \psi \right)dx.\label{eq3Prop12.2.1}
       \end{align}
    We estimate the last limsup in the right-hand side of \eqref{eq3Prop12.2.1}. First, 
    note that $\Omega\cap(\Omega_1\setminus\overline{A_1})\subset \Omega_1\cap\Omega_2$. Hence, $u^1_h-u^2_h$ converges to $0$ in $L^\infty(\Omega\cap(\Omega_1\setminus\overline{A_1}); \R^d)$.
         \begin{equation}
             \notag
             {t\over 1-t}(u^1_h-u^2_h)\otimes \nabla \psi \in B_\delta(0), 
         \end{equation}
    for any $h\geq h_t$ and for a.e.  $x\in\Omega\cap(\Omega_1\setminus\overline{A}_1)$. 
    In view of \eqref{hpL}, ${t\over 1-t}(u^1_h-u^2_h)\otimes \nabla \psi$  may be written as a convex combination of the vertices $\{A_i\}_{i=1}^{dn}$ of  cube $Q $     
    and
            \begin{align}
    &\int_{\Omega\cap(\Omega_1\setminus \overline{A_1})} g\left (\frac{x}{\e_h}, {t\over 1-t}(u^1_h-u^2_h)\otimes \nabla \psi\right)dx =\int_{\Omega\cap(\Omega_1\setminus \overline{A_1})} g\left(\frac{x}{\e_h}, \sum_{i=1}^{dn} s^h_i(x) A_{i}  \right)dx\notag\\
                &\leq \sum_{i=1}^{dn}  \int_{\Omega\cap(\Omega_1\setminus \overline{A_1})}  s^h_i(x) g\left(\frac{x}{\e_h}, A_i\right)dx\notag\\
                &\leq \sum_{i=1}^{dn}\int_{\Omega\cap(\Omega_1\setminus \overline{A_1})} g\left(\frac{x}{\e_h}, A_i\right)dx\notag
            \end{align}
              for any $h\geq h_t$. Now, the Riemann-Lebesgue Lemma ensures that there exists a positive constant $C_f=C(f)$ such that 
              \begin{equation}
                  \notag
                  \limsup_{h\to+\infty} \int_{\Omega\cap(\Omega_1\setminus \overline{A_1})}g\left(\frac{x}{\e_h}, {t\over 1-t}(u^1_h-u^2_h)\otimes \nabla \psi\right)dx\leq \mathcal L^n(\Omega)\sum_{i=1}^{{dn}}\int_Y g(y, A_i)dy\leq C_f.
              \end{equation}
            This along with \eqref{eq3Prop12.2.1} yields to 
               \begin{align}
                   \widetilde{G}''(\Omega, tu)& \leq \widetilde{G}''(\Omega_1, u) + \widetilde{G}''(\Omega_2, u)+(1-t)\mathcal L^n(\Omega)\int_Yg(y, 0)dy +(1-t) C_f,\notag
               \end{align}
            for any $t\in [0, 1)$. Finally, passing to the limit as $t\to 1$, we obtain estimate \eqref{ine12.2.4}, as desired.
\end{proof}

\subsection{Finiteness conditions}

In this subsection, we provide some sufficient conditions to get the finiteness of the functional $G''(\Omega, u)$. First, we show the finiteness on the set of piecewise affine functions.   The proof is quite technical and differs from the one available in the scalar setting where ad hoc cut-off functions maybe easily constructed (see \cite{CDA02}). Here we propose a more generic construction leading to a slight different estimate than the analogous one given in \cite[Lemma 12.3.1]{CDA02}.

For any $u\in \color{blue}A_{\rm pw}\color{black}(\R^n; \R^d)$, we set
      \begin{equation}
          \label{def:functionsigmau}
          \sigma(u)\coloneqq \max_{i\in\{1,\dots, m\}} {\rm card} \left\{j\in\{1,\dots, m\}\hspace{0.03cm}: \hspace{0.03cm} \overline{P}_j\cap \overline{P}_i \neq\emptyset   \right\}.
      \end{equation}
\begin{lemma}
\label{lemma12.3.1} Let $g: \R^n\times \R^{d \times n} \to [0, +\infty]$ satisfy {\rm (H1)}-{\rm (H4)}.  
\color{black} Let $G''$ be the functional given by \eqref{deflimsup} and $q\in [1,+\infty]$. 
Then, 
     \begin{equation}
            \notag
            {\color{blue}\widetilde{G}}''(\Omega, tu)\leq t\int_{\Omega}\widetilde{g}^q_\hom(Du)dx +(1-t){\cal L}^n(\Omega)\int_Y g(y,0)dy, 
        \end{equation}
    for any $\Omega\in{\color{blue} {\cal T}_0}$, $u\in \color{blue}A_{\rm pw}\color{black}(\R^n; \R^d)$ and  $t\in \biggl[0, {2\delta\over 4\sigma(u)(2\|Du\|_{L^\infty} +1)+\delta }\biggr]$, where $\widetilde{g}^q_\hom$ is defined by \eqref{ftildeCDA}.
\end{lemma}

\begin{proof}
  First, fix $\Omega\in{\color{blue} {\cal T}_0}$ and $u{(x)}=\sum_{i=1}^{m}(u_{Z_i}(x)+c_i)\chi_{P_i}(x)$, according to \eqref{aff-u-Zi}. For any $i\in\{1,\dots, m\}$, let $\Omega_i=\Omega\cap {\rm int}P_i$ be an open cover of $\Omega$. For $\e>0$ and $i\in\{1,\dots, m\}$, we introduce the following sets 
        \begin{align}
            \Omega_{i,\e}^{-} &\coloneqq \{x\in\Omega_i : {\rm dist}(x, \partial\Omega_i)>\e \},\notag\\
            \Omega_{i,\e}^{+} &\coloneqq \{x\in \mathbb{R}^n : {\rm dist}(x, \Omega_i)<\e \}.\notag
        \end{align}
    Fix $t\in [0,1)$ to be chosen later. Assume that 
         \begin{equation}
             \notag
             \sum_{i=1}^{m} \mathcal L^n(\Omega_i)\widetilde{g}^q_\hom(Z_i) = \int_{\Omega}\widetilde{g}^q_\hom(Du)dx<\infty.
         \end{equation}
    This implies that $Z_i\in{\rm dom}\widetilde{g}^q_\hom$ for every $i\in\{1,\dots, m\}$. Hence, for fixed $\theta\in (0, +\infty)$ and $i\in\{1, \dots, m\}$, there exists $v^i\in W^{1,q}_{\#}(Y; \R^d)\cap L^\infty(Y; \R^d)$ such that 
          \begin{equation}\label{rec}
              \int_Y g(y, Z_i+ Dv^i) dy \leq \widetilde{g}^q_\hom (Z_i) +\theta.
          \end{equation}
    For every $i\in\{1, \dots, m\}$, set $v^i_h(\cdot) = {1\over h} v^i(h\cdot)$, with $h\in\N$. It follows \color{blue}by Riemann-Lebesgue lemma and \eqref{rec} \color{black}that
        \begin{equation}
            \notag
            \lim_{h\to+\infty} \int_{\Omega\cap \Omega_{i, \e}^{+}} g\left({x\over \e_h}, Z_i+Dv^i_h\right)dx \leq {\cal L}^n(\Omega\cap \Omega^{+}_{i,\e}) (\widetilde{g}^q_\hom(Z_i) + \theta). 
        \end{equation}
    The aim is to approximate the fixed piecewise affine function $u\in \color{blue}A_{\rm pw}\color{black}(\R^n; \R^d)$ with a sequence of functions $\{w_{\e}\}_{\e}$ built by means of a suitable partition of unity subordinate to the open cover $\{\Omega_i\}_{i=1}^{m}$. We proceed in three steps: in the first one, we construct the partition of unity using an appropriate cut-off function and the family $\{w_\e\}_\e$, in the second one, we estimate the functional $G''(\Omega, t w_\e)$, while in the last step, we show our claim passing to the limit on $\varepsilon \to 0$.
    \medskip
    
    {\bf STEP 1.} For $\e>0$ sufficiently small and for $i\in\{1,\dots, m\}$, let $\psi_{i,\e}$ be a cut-off function such that $\psi_{i, \e}(x)\equiv 0$ a.e. in $\R^n \setminus \Omega_{i, \e}^{+}$, $\psi_{i, \e}(x)\equiv 1$ a.e. in $\overline{\Omega_{i, {\e\over 2}}^{+}}$, $0\leq \psi_{i, \e} (x)\leq 1$ a.e. in $\Omega_{i, \e}^{+}$ and 
        \begin{equation}
            \notag
            \|\nabla \psi_{i, \e} \|_{L^\infty(\R^n; \R^n)}\leq {C\over {\rm dist}(\Omega_{i, {\e\over 2}}^{+}, \partial \Omega_{i, \e}^{+})}.
        \end{equation}
    Note that 
       \begin{equation}
           \notag
           \sum_{i=1}^{m} \psi_{i, \e} (x)\geq 1 \qquad \mbox{for a.e.  } x\in \bigcup_{i=1}^{m} \Omega_{i, {\e\over 2}}^{+}.
       \end{equation}
    Then, for fixed $\e>0$ small enough, we define the partition of unity $\{\gamma_{i, \e}\}_{i=1}^{m}$ subordinate to $\{\Omega_i\}_{i=1}^{m}$ by
         \begin{equation}
             \notag
             \gamma_{i, \e}(x)\coloneqq {1\over \sum_{i=1}^{m}\psi_{i, \e} (x)}\psi_{i, \e} (x). 
         \end{equation}
    For fixed $\e>0$, let $\{w_{h, \e}\}_{h\in\N}$ be the sequence defined by 
        \begin{equation}
            \notag
            w_{h, \e}(x) \coloneqq \sum_{i=1}^{m} \left(u_{Z_i}(x) + c_i + v^i_h(x)\right)\gamma_{i, \e}(x).
        \end{equation}
    Since $v^i_h$ goes to $0$ in $L^\infty(\R^n; \R^d)$ as $h\to +\infty$,  we immediately get that, for every $\e>0$ sufficiently small,  $w_{h, \e}$ converges in $L^\infty(\R^n; \R^d)$ as $h\to+\infty$ to the function $w_\e$ given by  
         \begin{equation}
             \notag
             w_\e(x)\coloneqq \sum_{i=1}^{m} \left(u_{Z_i}(x)+ c_i\right) \gamma_{i, \e}(x).
         \end{equation}
    \medskip

    {\bf STEP 2.} Thanks to the convergence $w_{h, \e}\to w_{\e}$ as $h\to+\infty$ combined with the definition of $G''(\Omega, \cdot)$ and $(H2)$, for any $\e>0$ small enough, we deduce that  
        \begin{align}
            &G''(\Omega, tw_\e) \leq \limsup_{h\to+\infty} \int_{\Omega} g\left({x\over \e_h}, tDw_{h, \e}(x)\right) dx\notag\\
            &\quad=\limsup_{h\to+\infty} \int_{\Omega} g\biggl({x\over \e_h}, t\sum_{i=1}^{m} (Z_i+Dv^i_h(x))\gamma_{i, \e}(x) \notag\\
            &\hspace{5cm}+ (1-t){t\over (1-t)} \sum_{i=1}^{m} (u_{Z_i}(x)+c_i+v^i_h(x))\otimes \nabla \gamma_{i, \e}(x)\biggr) dx\notag
            \end{align}
            \begin{align}
            &\quad \leq t\limsup_{h\to+\infty} \int_{\Omega} g\left({x\over \e_h}, \sum_{i=1}^{m}(Z_i+Dv^i_h(x))\gamma_{i, \e}(x) \right)dx\notag\\
            &\qquad + (1-t) \limsup_{h\to+\infty} \int_{\Omega} g\left({x\over\e_h}, {t\over (1-t)} \sum_{i=1}^{m}(u_{Z_i}(x)+c_i+v^i_h(x))\otimes \nabla\gamma_{i, \e}(x)\right) dx\notag\\
            &\quad\leq t\sum_{i=1}^{m}\limsup_{h\to+\infty}\int_{\Omega\cap\Omega_{i, \e}^{+}} g\left({x\over \e_h}, Z_i+Dv^i_h(x) \right)dx\notag\\
            &\qquad  + (1-t) \limsup_{h\to+\infty} \int_{\Omega} g\left({x\over\e_h}, {t\over (1-t)} \sum_{i=1}^{m}(u_{Z_i}(x)+ c_i+v^i_h(x))\otimes \nabla\gamma_{i, \e}(x)\right) dx\label{eq1}
        \end{align}
    In view of Step 1, we known that $\sum_{i=1}^m \nabla \gamma_{i, \e}\equiv 0$ a.e. in $\Omega$ along with the fact that $\Omega= \left[\Omega\setminus\bigcup_{j=1}^m\Omega_{j,\e}^{-} \right]\cup \left[\bigcup_{j=1}^m\Omega_{j,\e}^{-} \right]$,  the last limsup in the right-hand side of \eqref{eq1} may be estimated as follows\\
         \begin{align}
           \limsup_{h\to+\infty} &\int_{\Omega} g\left({x\over\e_h}, {t\over (1-t)} \sum_{i=1}^{m}(u_{Z_i}(x)+c_i+v^i_h(x))\otimes \nabla  \gamma_{i, \e}(x)\right)  \notag\\
           &=\limsup_{h\to+\infty}\int_{\Omega}g\left({x\over \e_h}, {t\over (1-t)} \sum_{i=1}^{m}(u_{Z_i}(x)+c_i+v^i_h(x)-u(x))\otimes \nabla  \gamma_{i, \e}(x)\right) dx\notag\\
           &\leq \sum_{j=1}^{m}\limsup_{h\to+\infty}\int_{\Omega_{j, \e}^{-}   } g\left({x\over \e_h}, 0 \right)dx\notag\\
           &\quad + \sum_{j=1}^{m}\limsup_{h\to+\infty}\int_{\Omega_j\setminus\Omega_{j, \e}^{-}}g\left({x\over \e_h} ,{t\over (1-t)} \sum_{i=1}^{m}(u_{Z_i}(x)+c_i+v^i_h(x)-u(x))\otimes \nabla \gamma_{i, \e}(x)\right)dx. \label{eq2}
         \end{align}
    Now, for any $j\in\{1, \dots, m\}$, let $\nu_\e(\Omega_1, \dots, \Omega_m)(\Omega_j)$ be the number of the elements $\{\Omega_1, \dots, \Omega_m \}$ whose distance from $\Omega_j$ is less than $\e$.  Then, we define $\sigma_\e=\sigma_\e(\Omega_1, \dots, \Omega_m)$ by
        \begin{equation}
            \notag
            \sigma_\e\coloneqq \sup_{j\in\{1,\dots, m\}} \nu_\e(\Omega_1, \dots, \Omega_m)(\Omega_j).
        \end{equation}
    Since for {\color{blue} any} fixed $\e>0$ and for any $j\in\{1,\dots, m\}$, $\nu_\e(\Omega_1, \dots, \Omega_m)(\Omega_j)\leq \sigma_\e$, we denote by $\{\Omega_{i_1}, \dots, \Omega_{i_{\sigma_\e}}\}$ a subset of $\{\Omega_1,\dots, \Omega_m\}$ containing all the sets $\Omega_i$ satisfying ${\rm dist}(\Omega_j, \Omega_i)<\e$. Therefore,  due to the convexity of $g$, we deduce that, for any $j\in\{1,\dots,m\}$ and for any $\e>0$ sufficiently small,
        \begin{align}
            &\int_{\Omega_j\setminus\Omega_{j, \e}^{-}}g\left({x\over\e_h}, {t\over (1-t)} \sum_{i=1}^{m}(u_{Z_i}(x)+c_i+v^i_h(x)-u(x))\otimes \nabla  \gamma_{i, \e}(x)\right)dx\notag\\
            &\quad =\int_{\Omega_j\setminus\Omega_{j, \e}^{-}}g\left({x\over\e_h}, {t\over (1-t)} \sum_{k=1}^{\sigma_\e}{1\over \sigma_\e}\sigma_\e(u_{Z_{i_k}}(x)+c_{i_k}+v^{i_k}_h(x)-u(x))\otimes \nabla \gamma_{i_k, \e}(x)\right)dx\notag
            \end{align}
            \begin{align}
            &\quad\leq\sum_{k=1}^{\sigma_\e}{1\over \sigma_\e}\int_{\Omega_j\setminus\Omega_{j, \e}^{-}}g\left({x\over \e_h}, {t\over (1-t)} \sigma_\e(u_{Z_{i_k}}(x)+c_{i_k}+v^{i_k}_h(x)-u(x))\otimes \nabla \gamma_{i_k, \e}(x)\right)dx\notag\\
            &\quad \leq \sum_{k=1}^{\sigma_\e}{1\over \sigma_\e}\int_{(\Omega_j\setminus\Omega_{j, \e}^{-})\cap\Omega^{+}_{i,\e}}g\left({x\over \e_h}, {t\over (1-t)} \sigma_\e(u_{Z_{i_k}}(x)+c_{i_k}+v^{i_k}_h(x)-u(x))\otimes \nabla \gamma_{i_k, \e}(x)\right)dx\notag\\
            &\qquad +\int_{\Omega_j\setminus\Omega^{-}_{j, \e}} g\left({x\over\e_h}, 0\right)dx.\notag
        \end{align}
    Now, note that there exists $\e(u)\in(0, +\infty)$ such that $\sigma_\e\leq \sigma(u)$ for any $\e\in(0, \e(u))$, where $\sigma(u)$ is given by \eqref{def:functionsigmau}.  Fix $\e\in(0, \e(u))$. Since, by \color{blue} Taylor's expansion and the definition of $v^i_h$, we obtain that\color{black} 
        \begin{align}
            \left\| \sigma_\e{t\over 1-t}(u_{Z_{i}}(\cdot)+ c_i+v^i_h(\cdot)-u(\cdot))  \right\|_{L^\infty(\Omega\cap\Omega^{+}_{i,\e}; \R^{d\times n}) }\notag \\\leq \sigma_\e {t\over 1-t}\left(2\|Du\|_{L^\infty(\Omega\cap\Omega^{+}_{i,\e}; \R^{d\times n}) } +1\right)\e\notag\\
            \leq \sigma(u) {t\over 1-t}\left(2\|Du\|_{L^\infty(\Omega\cap\Omega^{+}_{i,\e}; \R^{d\times n}) } +1\right)\e,\label{torefer}
        \end{align}
    we may choose $t\in [0,1)$ such that 
          \begin{equation}
              \notag
              \sigma(u){t\over 1-t} \left(2\|Du\|_{L^\infty(\Omega\cap\Omega^{+}_{i,\e}; \R^{d\times n}) } +1\right) \leq{\delta\over 2},
          \end{equation}
     with $\delta$ being given by assumption (H3). In other words, 
        \begin{equation}
            \notag
            t\leq 
            {2\delta \over \sigma(u)\left(2\|Du\|_{L^\infty(\Omega\cap\Omega^{+}_{i,\e}; \R^{d\times n}) } +1\right) +\delta}.
        \end{equation}
    Hence, by \eqref{torefer}, for this choice of $t$, we obtain that 
        \begin{align}\notag
            \left\| \sigma_\e{t\over 1-t}(u_{Z_{i}}(\cdot)+c_i+v^i_h(\cdot)-u(\cdot))  \right\|_{L^\infty(\Omega\cap\Omega^{+}_{i,\e};\R^{d \times n})} \leq{\delta\over 2}\e,
        \end{align}
    for $h$ sufficiently large and for any $i\in\{1, \dots, m\}$. This ensures that there exists $h_t$ such that 
       \begin{align}
           \notag
           \sigma_\e{t\over 1-t}(u_{Z_{i}}(x)+c_i+v^i_h(x)-u(x))\otimes \nabla \gamma_{i, \e} \in B_{\delta}(0),
       \end{align}
    for any $h>h_t$ and for any $\e>0$ sufficiently small. Thanks to \eqref{hpL}, for any $i\in\{1, \dots, m\}$, $\sigma_\e{t\over 1-t}(u_{Z_{i}}+c_i + v^i_h-u)\otimes \nabla \gamma_{i, \e}$ may be represented as a convex combination of the vertices $\{A_l\}_{l=1}^{{dn}}$ of the cube $Q$, i.e., for any $i\in\{1,\dots, m\}$ there exist $s^{i,h}_1(x), \dots, s^{i,h}_{dn}(x)\in(0,1)$ such that 
        \begin{equation}
            \notag
            \sigma_\e{t\over 1-t}(u_{Z_{i}}+ c_i +v^i_h-u)\otimes \nabla \gamma_{i, \e} =\sum_{l=1}^{dn} s^{i,h}_l A_l.
        \end{equation}
      
    Therefore, for any $\e\in(0, \e(u))$, once more by exploiting the convexity of $g$, we deduce
        \begin{align}
            &\limsup_{h\to+\infty} \sum_{k=1}^{\sigma_\e}{1\over \sigma_\e}\int_{(\Omega_j\setminus\Omega_{j, \e}^{-})\cap\Omega^{+}_{j,\e}}g\left({x\over \e_h}, {t\over (1-t)} \sigma_\e(u_{Z_{i_k}}(x)+ c_{i_k}+v^{i_k}_h(x)-u(x))\otimes \nabla  \gamma_{i_k, \e}(x)\right)dx\notag\\
            &\quad = \limsup_{h\to+\infty} \sum_{k=1}^{\sigma_\e}{1\over \sigma_\e} \int_{(\Omega_j\setminus\Omega_{j, \e}^{-})\cap\Omega^{+}_{j,\e}}g\left({x\over \e_h}, \sum_{l=1}^{{dn}} s_l^{i_k, h_k}(x) A_l \right)dx\notag\\
            &\quad \leq \limsup_{h\to+\infty} \sum_{k=1}^{\sigma_\e}{1\over \sigma_\e}  \sum_{l=1}^{{dn}} \int_{(\Omega_j\setminus\Omega_{j, \e}^{-})\cap\Omega^{+}_{j,\e}} s_l^{i_k,h_k}(x)g\left({x\over \e_h}, A_l \right)dx\notag\\ 
            &\quad \leq {\cal L}^n(\Omega_j\setminus\Omega_{j,\e}^{-}) \sum_{l=1}^{{dn}}\int_Y g(y, A_l)dy.\notag
        \end{align}
   This combined with \eqref{eq2} implies that  
         \begin{align}
             \limsup_{h\to+\infty} &\int_{\Omega} g\left({x\over\e_h}, {t\over (1-t)} \sum_{i=1}^{m}(u_{Z_{i}}(x)+c_i+v^i_h(x))\otimes \nabla \gamma_{i, \e}(x)\right)  \notag\\
           &\leq \sum_{j=1}^{m}\limsup_{h\to+\infty}\int_{ \Omega_{j, \e}^{-} } g\left({x\over \e_h}, 0 \right)dx\notag\\
           &\quad + \sum_{j=1}^{m}\limsup_{h\to+\infty}\int_{\Omega_j\setminus\Omega_{j, \e}^{-}}g\left({x\over \e_h} ,{t\over (1-t)} \sum_{i=1}^{m}(u_{Z_{i}(x)+c_i}+v^i_h(x)-u(x))\otimes \nabla  \gamma_{i, \e}(x)\right)dx\notag\\
           &\quad \leq\sum_{j=1}^{m}\left[ {\cal L}^n(\Omega_{j, \e}^{-}) \int_Y g(y, 0)dy +  {\cal L}^n(\Omega_j\setminus\Omega_{j,\e}^{-}) \sum_{l=1}^{{dn}}\int_Y g(y, A_l)dy\right]\notag.
         \end{align}
    Hence, from \eqref{eq1}, we conclude that 
        \begin{align}
            G''(\Omega, tw_\e) &\leq t\sum_{i=1}^{m}{\cal L}^n(\Omega\cap\Omega_{i,\e}^{+})(\widetilde{g}^q_\hom(Z_i) + \theta)\notag\\
            &\quad + (1-t) \sum_{j=1}^{m} \left[{\cal L}^n(\Omega_{j,\e}^{-})  \int_Y g(y, 0) dy+ {\cal L}^n(\Omega{\color{blue}_j}\setminus\Omega_{j,\e}^{-}) \sum_{l=1}^{{dn}} \int_Y g(y, A_l)dy\right], \label{eq3}
        \end{align}
    for any $\e\in(0, \e(u))$, concluding the Step 2.
    \medskip
    
    {\bf STEP 3.} First, note that, for any $\e\in(0, \e(u))$,
          \begin{align}
              \|w_\e-u\|_{L^\infty(\Omega; \R^d)}&=\left \|\sum_{i=1}^{m}(u_{Z_{i}}(\cdot)+c_i-u(\cdot))\gamma_{i,\e}(\cdot) \right\|_{L^\infty(\Omega; \R^d)}\notag\\
              &\leq \sum_{i=1}^{m}\|(u_{Z_{i}}(\cdot) +c_i-u(\cdot))\gamma_{i,\e}(\cdot)\|_{L^\infty(\Omega; \R^d)}\notag\\
              &\leq m\left(2\|Du\|_{L^\infty(\Omega; \R^{d\times n})}+1\right)\e,\notag
          \end{align}
    This implies that $w_\e$ converges in $L^\infty(\Omega; \R^d)$ to $u$ as $\e\to 0$. Hence, using estimate \eqref{eq3}, we deduce that 
          \begin{align}
              G''(\Omega, tu) &\leq \liminf_{\e\to 0} G''(\Omega, tw_\e)\notag\\
              &\leq t\sum_{i=1}^{m}{\cal L}^n(\Omega\cap\overline{\Omega_{i}})(\widetilde{g}^q_\hom(Z_i) + \theta)+(1-t){\cal L}^n(\Omega)\int_Yg(y, 0) dy.\notag
          \end{align}
    Taking the limit as $\theta\to 0$, we obtain our claim, as desired.
\end{proof}

In light of Lemma \ref{lemma12.3.1}, we show the finiteness result in the next proposition.

\begin{prop}
\label{Prop 12.3.2}
    Let $g: \R^n\times \R^{d \times n} \to [0, +\infty]$ satisfy {\rm (H1)}-{\rm (H4)}.  
    Let $G''$ be the functional given by \eqref{deflimsup} and $q\in[1,+\infty]$. Then, there exists $r\in (0, \delta)$ and a constant $C>0$ such that 
        \begin{equation}
            \notag
            G''_{-}(\Omega, u)\leq C {\cal L}^n(\Omega),
        \end{equation}
    for any $\Omega\in{\color{blue} {\cal T}_0}$ and for any $u\in W^{1,\infty}_{\#} (Y; \R^d)$ such that $\|Du\|_{L^\infty(\Omega; \R^{d \times n})}\leq r$.
\end{prop}

\begin{proof}
    Fixed $\Omega\in{\color{blue} {\cal T}_0}$ and let $Q$ be an open cube of $\R^n$ such that $\Omega\subset\subset Q$. Let $r\in (0,+\infty)$ to be chosen later and let $u\in W^{1,\infty}_{\#}(\R^n; \R^d)$ such that $\|Du\|_{L^\infty(\Omega;\R^{d \times n})}\leq r$.\\
    It is not restrictive to assume that $u=0$ in $\R^n\setminus Q$. For some $l\in\N$, let $S_1, \dots, S_l\subset \R^n\setminus Q$ be polyhedral sets with pairwise disjoint interiors such that ${\cal L}^n((\R^n\setminus Q)\setminus \bigcup_{i=1}^{l}S_i)=0$. Let $P_1, \dots, P_m\subset Q$ be $n$-simplexes with pairwise disjoint interiors such that $Q=\bigcup_{i=1}^{m} P_i$. For every $h\in\N$, let $P_1^h, \dots, P^h_{m^n}$ be $n$-simplexes obtained by taking the ${1\over h}$-replies of $P_1, \dots, P_m$ repeated ${1\over h}Q$-periodically so that $Q=\bigcup_{i=1}^{m^n}P^h_i$.
    For any $h\in\N$, let $u_h\in \color{blue}A_{\rm pw}\color{black}(\R^n; \R^d)$ be such that $u_h$ is affine on each $n$-simplex of $\{P_1^h, \dots, P_{m^n}^h\}$, equal to $u$ on the vertices of the elements of $\{P_1^h, \dots, P_{m^n}^h\}$ and equal to $0$ in each element of $\{S_1, \dots, S_l\}$. Since, for any $h\in\N$ and for any $i\in\{1, \dots, m^n\}$, $P_i^h$ intersects at most $m^n$ elements of $\{P_1^h, \dots, P_{m^n}^h\}$, we get that 
        \begin{itemize}
            \item[(i)] $\sigma(u_h)\leq m^n+l$, with $\sigma(u)$ being given by \eqref{def:functionsigmau};
            \item[(ii)] $u_h$ converges in $L^\infty(\R^n; \R^d)$ to $u$ as $h\to\infty$;
            \item[(iii)] $\|Du_h\|_{L^\infty(\Omega; \R^{d\times n})}\leq C\|Du\|_{L^\infty(\Omega; \R^{d\times n})}$ for any $h\in\N$.
        \end{itemize}
        
    Hence, by items (i) and (iii), we deduce that 
       \begin{equation}
           \notag
           {\delta\over 4(m^n+l)(2\delta+1)+\delta} \leq {\delta\over 4\sigma(u_h)\biggl(\|{\delta\over Cr} Du_h\|_{L^\infty(\Omega; \R^{d\times n})} +1 \biggr)+\delta}, 
       \end{equation}
       where $C$ is the constant appearing in (iii).
    An application of Lemma \ref{lemma12.3.1} yields to  
          \begin{equation}
              \label{eq4}
              G^{''} \left (\Omega, t{\delta\over C r} u_h \right ) \leq t\int_{\Omega}\widetilde{g}^q_\hom\left({\delta\over Cr}Du_h\right)dx+(1-t){\cal L}^n(\Omega)\int_Y g(y,0)dy,               
          \end{equation}
    for any $h\in\N$ and for any $t\in\left[0, {\delta\over 4(m^n+l)(2\delta+1)+\delta} \right]$. 
    We can pick $t={Cr\over\delta}$ if and only if 
       \begin{equation}
           \notag
           r\leq {2\delta^2\over c\left[4(m^n+l)(2\delta+1)+\delta\right]}.
       \end{equation}
    Thus, choosing $r$ as above in \eqref{eq4}, we deduce that, for any $h\in\N$,
       \begin{equation}
       \label{eq5}
           G''(\Omega, u_h)\leq {Cr\over\delta} \int_{\Omega} \widetilde{g}^q_\hom\left( {\delta\over Cr}Du_h\right) + \left(1-{Cr\over\delta}\right) {\cal L}^n(\Omega) \int_Y g(y,0)dy. 
       \end{equation}
    From the fact that $\|Du_h\|_{L^\infty(\Omega; \R^{d\times n})}\leq Cr$, it follows that 
         \begin{equation}
             \label{eq6}
             \left\| {\delta\over Cr} Du_h\right\|_{L^\infty(\Omega; \R^{d\times n})}\leq \delta \qquad\mbox{for any } h\in\N,
         \end{equation}
    namely, for any $h\in\N$, ${\delta\over Cr} Du_h\in B_\delta(0)$. This combined with (H3) ensures us that $\widetilde{g}^q_\hom$ is bounded in $B_\delta(0)$. Finally, from \eqref{eq5} and \eqref{eq6}, we conclude that 
         \begin{equation}
             \notag
             G''(\Omega, u_h) \leq \left[\max_{\overline{B_\delta(0)}}\widetilde{g}^q_\hom  + \int_Y g(y,0)dy \right]{\cal L}^n(\Omega) .
         \end{equation}
    Since $u_h$ converges to $u$ as $h\to+\infty$ and in view of the lower semicontinuity of $G''$, it follows that 
          \begin{align}
              G''(\Omega, u)&\leq \limsup_{h\to+\infty} G''(\Omega, u_h)\notag\\
              &\leq \left[\max_{\overline{B_\delta(0)}}\widetilde{g}^q_\hom + \int_Y g(y,0)dy\right]{\cal L}^n(\Omega),\notag             
          \end{align}
    which concludes the proof.
    \end{proof}

\subsection{Representation on Linear Functions and a blow-up condition}


\color{black} In this subsection, we aim to show that, for any bounded open subset $\Omega$ of $\R^n$,  it holds that $G'(\Omega, \cdot) = G''(\Omega, \cdot)$ on the space of the linear functions and it possible to represent their common value as an integral.

\par The next proposition claims an upper bound for the $\Gamma$-limsup on the space of linear functions and it is the analogue of \cite[Lemma 12.4.1]{CDA02} in the vectorial case. The proof follows the same arguments as the one in the scalar setting and for this reason, we skip it.

\begin{prop}
    \label{Lemma 12.4.1} 
    Let $g: \mathbb R^n \times \mathbb R^{d\times n} \to [0,+\infty]$ satisfy {\rm (H1)}-{\rm (H2)}, let $\widetilde{g}^q_\hom$ be defined in \eqref{ftildeCDA} and let $\{G_{\varepsilon}\}_\varepsilon$ be the functionals  defined by \eqref{functCDA}, with $G''$ as in \eqref{deflimsup}.  \color{black}
     We have that
         \begin{equation}
             \notag
             G''(\Omega, u_Z)\leq {\mathcal L}^n(\Omega)(\widetilde{g}^q_\hom)^{\rm ls}(Z),
         \end{equation}
for any $\Omega\in {\color{blue} {\cal T}_0}$ and $Z\in\R^{d\times n}$.
\end{prop}
To prove a similar inequality with $G'$ in place of $G''$, we gather some properties whose proofs are skipped since they follow the same lines of \cite[Lemma 12.4.2, 12.4.3 and 12.4.4]{CDA02} respectively, besides now we are in the vectorial setting.
We just observe that as in \cite[Lemma 12.4.2]{CDA02}, the proof of (i) below is a consequence of the convexity of $g(x,\cdot)$ inherited by $G'$ and $G''$ (see \cite[Proposition 3.4,1.]{CDA02}) also in the vectorial case.
\color{black}

   \begin{lemma}
       \label{Lemma 12.4.2, 12.4.3, 12.4.4}
       Let $g: \R^n\times \R^{d \times n} \to [0, +\infty]$ satisfy {\rm (H1)} and {\rm (H2)}.  Let $\{G_{\varepsilon}\}_\varepsilon$ be the functionals  defined by \eqref{functCDA} and let $G'$ be as in \eqref{defliminf}. \color{black} 
       Then, 
          \begin{itemize}
              \item[{\rm (i)}] assuming also \eqref{H40}, 
                  \begin{equation}
                  \notag
                      G'(\Omega, tu)\leq t G'(\Omega, u) + (1-t) {\mathcal L}^n(\Omega) \int_Y g(y, 0)dy,
                  \end{equation}
                  for any $\Omega\in {\color{blue} {\cal T}_0}$, $u\in L^\infty_{\loc}(\R^n; \hspace{0.03cm} \R^d)$ and for any $t\in [0,1]$.  Similar inequalities hold for $G'', G'_{-}, G''_{-}$ in place of $G'$, with $G''$ defined in \eqref{deflimsup}, and the latter two functionals being the inner regular envelopes of the previous ones.
              \item[{\rm (ii)}]
                  \begin{equation}
                      \notag
                      {1\over r_1^n} G'(x_1+r_1 Y, u_Z)={1\over r_2^n} G'(x_2+r_2 Y, u_Z),
                  \end{equation}
                for any $x_1, x_2\in\R^n$, $r_1, r_2\in (0, +\infty)$, and $Z\in\R^{d\times n}$.
            \item[{\rm (iii)}]  
                 \begin{equation}
                     \notag
                     \widetilde{g}^q_\hom (Z) = \inf\left\{\int_Y g(hx, Z+Du)dx \hspace{0.03cm}: \hspace{0.03cm} v\in W^{1,q}_{\#}(Y; \hspace{0.03cm} \R^d)\cap L^\infty(Y; \hspace{0.03cm} \R^d)    \right\},
                 \end{equation}
                 for any $Z\in\R^{d\times n}$ and $h\in\mathbb{N}$, with $\widetilde{g}^q_{\hom}$ defined by \eqref{ftildeCDA} and $q\in[1,+\infty]$.
          \end{itemize}
   \end{lemma}

\color{black}
 The following lemma is the analogue of  \cite[Lemma 12.4.5]{CDA02} in the vectorial case. However, the proof requires some technical changes and for the readers' convenience, we present {\color{blue} it}.
 
\begin{lemma}
    \label{Lemma 12.4.5}
    Let $g: \R^n\times \R^{d \times n} \to [0, +\infty]$  satisfy {\rm (H1)}-
    {\rm (H4)}.
    Let $G'$ be the functional defined by \eqref{defliminf} such that  \begin{equation}
        \label{ass-F'meno12}
    G'((-1, 2)^n, u_Z)< + \infty,
    \end{equation}
    for some $Z\in\R^{d\times n}$. Then,
         \begin{equation}
         \notag
             \widetilde{g}^q_\hom(tZ)<+\infty \qquad\mbox{for all } t\in [0,1).
         \end{equation}
\end{lemma}
\begin{proof}
   In view of \eqref{ass-F'meno12}
   there exists a sequence $\{v_h\}_h\subseteq W^{1,q}_{\loc}(\R^n; \R^d)\cap L^\infty_{\loc}(\R^n; \R^d)$ and $\{h_k\}\subset\N$ strictly increasing such that $v_h$ converges to $u_Z$ in $L^\infty((-1, 2)^n; \R^d)$ and
           \begin{equation}
               \label{ref2}
               \int_{(-1,2)^n} g\left({x\over \e_{h_k}}, Dv_{h_k}\right)dx<+\infty\qquad\mbox{for any } k\in\N.
           \end{equation}
    For fixed $\eta\in(0,1)$, let $\psi$ be a  \color{blue} smooth \color{black} cut-off function such that $\psi\equiv 1$ in {\color{blue} $Y$}, $0\leq\psi\leq 1$ in $(-\eta, 1+\eta)^n$ and $\psi\equiv 0$ in $\R^n\setminus (-\eta, 1+\eta)^n$ and 
        \begin{equation}
            \label{propcutofffunction}
            \|\nabla \psi\|_{L^\infty(\R^n; \R^d)} \leq {C\over {\rm dist}(Y, \partial (-\eta, 1+\eta)^n)}.
        \end{equation}  
    Note that $\sum_{i\in\Z^n} \psi(x+i)$ is a sum over a finite set of indices and the following inequality
          \begin{equation}
              \notag
              \sum_{i\in\Z^n} \psi(x+i)\geq 1 \qquad \mbox{for a.e. } x\in\R^n,
          \end{equation}
    holds. Set 
       \begin{equation}
           \notag
           \widetilde{\psi}(x) \coloneqq {\psi(x)\over \sum_{j\in\Z^n}\psi(x+j)} \qquad\mbox{for a.e. } x\in\R^n.
       \end{equation}
   Note that $\widetilde{\psi}\equiv 0$ in $\R^n \setminus (-\eta, 1+\eta)^n$, $0\leq\widetilde{\psi}\leq 1$ in $\R^n$ and $\sum_{i\in\Z^n}\widetilde{\psi}(x+i)=1$. We define the sequence $\{u_h\}_h$ by
       \begin{equation}
           \notag
           u_h(x) \coloneqq u_Z(x) + \sum_{i\in\Z^n} (v_h(x+i)-u_Z(x+i)) \widetilde{\psi}(x+i),
       \end{equation}
   for a.e. $x\in\R^n$ and every $h\in\N$. We stress out that the above sum is extended only to a finite set of indices $i\in\Z^n$ and, as a by-product, $u_h\in W^{1, q}_{\loc}(\R^n; \R^d)\cap L^\infty_{\loc} (\R^n; \R^d)$. In particular, one can show that $u_h\in W^{1, q}_{\#}(Y; \R^d)\cap L^\infty_{\loc} (Y; \R^d)$. For fixed $t\in [0,1)$, we  prove that there exists $k_t\in \N$ such that
        \begin{equation}
            \label{eq1.4}
            \int_Y g\left({x\over \e_{h_{k}} }, tZ+tD(u_{h_{k}}(x) - u_Z(x))    \right)dx<+\infty,
        \end{equation}
    for any $k\geq k_t$.
    To that end, making use of the convexity properties of $g$ as well as of the fact that $\sum_{i\in\Z^n}\widetilde{\psi}(x+i)=1$, we deduce that  
       \begin{align}
           &\int_Y g\left({x\over \e_{h_{k}}}, tZ+tD(u_{h_{k}}-u_Z)\right)dx\notag\\
           &=\int_Y g\left({x\over \e_{h_{k}}}, tZ+tD\left(\sum_{i\in\mathbb{Z}^n}(v_{h_k}(x+i) -u_Z(x+i))\widetilde{\psi}(x+i)  \right)   \right)dx\notag\\
           & =\int_Y g\biggl({x\over \e_{h_{k}}}, tZ+t\sum_{i\in\mathbb{Z}^n}\widetilde{\psi}(x+i) D(v_{h_k}(x+i)-u_Z(x+i))\notag\\
           &\hspace{5cm} + t\sum_{i\in\mathbb{Z}^n}(v_{h_k}(x+i)-u_Z(x+i))\otimes \nabla\widetilde{\psi}(x+i)  \biggr)dx\notag\\
           &=\int_Y g\left({x\over \e_{h_{k}}}, t\sum_{i\in\mathbb{Z}^n} \left(\widetilde{\psi}_{h_k}(x+i)Dv_{h_k}(x+i)+ (v_{h_k}(x+i)-u_Z(x+i))\otimes \nabla\widetilde{\psi}(x+i)\right) \right)dx\notag\\
           &\leq t\underbrace{\int_Y g\left({x\over \e_{h_{k}}}, \sum_{i\in\mathbb{Z}^n} \widetilde{\psi}_{h_k}(x+i)Dv_{h_k}(x+i) \right) dx}_{\eqqcolon {\cal I}_1}\notag\\
           &\quad + (1-t) \underbrace{\int_Y g\left({x\over \e_{h_{k}}}, {t\over 1-t} \sum_{i\in\mathbb{Z}^n}(v_{h_k}(x+i)-u_Z(x+i))\otimes \nabla\widetilde{\psi}(x+i) \right)dx}_{\eqqcolon {\cal I}_2},\label{eq1.3}
       \end{align}
   for any $k\in\N$.
Now we estimate ${\cal I}_1$ and ${\cal I}_2$. To that end, set
    \begin{equation}
        \notag
        I\coloneqq \{i\in\Z^n\hspace{0.03cm}:\hspace{0.03cm} (Y+i)\cap (-\eta, 1+\eta)^n\neq \emptyset \}.
    \end{equation}
Note that $I$ has exactly $3^n$ elements and ${\mathcal L}^n((-1, 2)^n\setminus \bigcup_{i\in I}(Y+i))=0$. Moreover, we have that 
    \begin{equation}
        \notag
        \sum_{i\in I} \widetilde{\psi}_h (x+i)=1 \quad\mbox{for a.e. } x\in Y, \mbox{ and for any }h\in\mathbb{N}.
    \end{equation}
This implies the finiteness ${\cal I}_1$. Indeed, by $\rm (H2)$
    \begin{align}
        {\cal I}_1&=\int_Y g\left({x\over \e_{h_k}}, \sum_{i\in I} \widetilde{\psi}(x+i)Dv_{h_k}(x+i) \right) dx\notag\\
        &\leq \sum_{i\in I} \int_Y \widetilde{\psi}(x+i)g\left({x\over \e_{h_k}}, Dv_{h_k}(x+i) \right) dx\notag\\
        &\leq  \sum_{i\in I} \int_Y g\left({x\over \e_{h_k}}, Dv_{h_k}(x+i) \right) dx\notag\\
        &=  \sum_{i\in I} \int_{Y+i} g\left({y-i\over \e_{h_k}}, Dv_{h_k}(y) \right) dy\notag\\
        &= \sum_{i\in I} \int_{Y+i} g\left({y\over \e_{h_k}}, Dv_{h_k}(y) \right) dy\notag\\
        &=\int_{(-1,2)^n} g\left({y\over \e_{h_k}}, Dv_{h_k}(y) \right) dy<+\infty,\label{eq1.1}
    \end{align}
where in the last line we made use of \eqref{ref2}.
  
Now, we estimate ${\cal I}_2$.  First, note that $\nabla \widetilde{\psi} (x+i) = 2\lambda^i(x)\nabla\psi(x+i) - 2 \mu^i(x)\sum_{j\in\Z^n}\nabla \psi(x+i)$, where $\lambda^i$ and $\mu^i$ are defined respectively by 
    \begin{equation}
        \notag
        \lambda^i(x)\coloneqq {1\over 2\sum_{j\in\Z^n} \psi(x+i+j)}, \qquad \mu^i(x)\coloneqq {\psi(x+i)\over 2\left(\sum_{j\in\Z^n} \psi(x+i+j)\right)^2}.
    \end{equation} 
It \color{blue} results \color{black} that $0\leq\lambda^i(x)\leq {1\over 2}$ and $0\leq\mu^i(x)\leq {1\over 2}$ for a.e. $x\in Y$, for any $i\in I$.

Therefore, the convexity of $g$ leads us to
    \begin{align}
        {\cal I}_2&\leq \sum_{i\in I}{1\over 3^n}\int_Y g\left({x\over\e_{h_k}}, 3^n{ t\over 1-t} (v_{h_k}(x+i)-u_Z(x+i))\otimes \nabla\widetilde{\psi}(x+i) \right)dx\notag\\
        &\leq \sum_{i\in I} \int_Y  g\left({x\over\e_{h_k}}, 3^n{t\over 1-t} (v_{h_k}(x+i)-u_Z(x+i))\otimes2 \nabla \psi(x+i)\right) dx\notag\\
        &\quad + \sum_{i\in I} \int_Y  g\left({x\over\e_{h_k}}, -3^n{ t\over 1-t} 2(v_{h_k}(x+i)-u_Z(x+i))\otimes \sum_{j\in\Z^n}\nabla\psi(x+i+j)\right) dx\notag\\
        &\quad + \sum_{i\in I} \int_Y g\left({x\over\e_{h_k}}, 0\right)dx.\label{eq1.0}
    \end{align}

Now, we consider the first integral in the right-hand side of \eqref{eq1.0}. The periodicity of $g$ yields to 
    \begin{align}
        \int_Y & g\left({x\over\e_{h_k}}, 3^n {t\over 1-t} (v_{h_k}(x+i)-u_Z(x+i))\otimes2 \nabla \psi(x+i)\right) dx\notag\\
        & =\int_{Y+i}g\left({y\over\e_{h_k}}, 3^n{ t\over 1-t} (v_{h_k}(y)-u_Z(y))\otimes2 \nabla \psi(y)\right) dy.\notag
    \end{align}
Since $Y+i\subset (-1, 2)^n$ along with the fact that $v_h\to u_z$ in $L^\infty((-1, 2)^n; \hspace{0.03cm}\R^d)$  as $h\to+\infty$ and estimate \eqref{propcutofffunction}, for any $t\in(0,1)$ there exists $k_{t}\in\N$ such that 
     \begin{equation}
         \notag
         3^n2 {t\over 1-t}(v_{h_k}-u_Z)\otimes \nabla \psi\in B_\delta (0) \qquad\mbox{for } k\geq k_{ t}.
     \end{equation}
By \eqref{hpL}, we may write $3^n2 {t\over 1-t}(v_{h_k}-u_Z)\otimes\nabla \psi$ as a convex combination of the vertices $\{A_i\}_{i=1}^{dn}$ of the cube $Q$, i.e., there exist, for every $h_k \in \mathbb N$, functions $s^{h_k}_1,\dots, s^{h_k}_{nd}: \Omega \to (0,1)$ such that
      \begin{equation}
          \notag
           3^n2 {t\over 1-t} (v_{h_k}-u_Z)(x)\otimes\nabla \psi(x) = \sum_{l=1}^{nd} s^{h_k}_l(x) A_l. 
      \end{equation}

Thus, in view of the periodicity of $g$ in the first variable and \eqref{hpL},
     \begin{align}
         \int_{Y+i}g\left({y\over\e_{h_k}}, 3^n{ t\over 1-t} (v_{h_k}(y)-u_z(y))\otimes2\nabla \psi(y)\right) dy &\leq \sum_{l=1}^{nd} \int_{Y+i} g\left({y\over\e_{h_k}}, A_l\right)dy<+\infty,\notag 
     \end{align}
     
\noindent for any $k\geq k_t$, concluding that the first integral in the right-hand side of \eqref{eq1.0} is finite. \\
Using similar arguments (see \cite[proof of Lemma 12.4.5]{CDA02}), and exploit the periodicity and convexity of $g$ in {\rm (H2)}, together with the convergence of $v_h$ towards $u_Z$ and \eqref{hpL}, it is possible to show that also the second integral in the right-hand side of \eqref{eq1.0} is finite, so that from \eqref{eq1.0}, it follows that
     \begin{align}
         {\cal I}_2&\leq \sum_{i\in I} \left[2\sum_{l=1}^{nd}\int_{Y-i} g\left({x\over \e_{h_k}}, A_l \right)dy + \int_Y g\left({x\over \e_{h_k}}, 0 \right)dy\right]<+\infty,\label{eq1.2}
     \end{align}
for any $k\geq k_t$. Gathering \eqref{eq1.1} and \eqref{eq1.2}, from \eqref{eq1.3}, we obtain \eqref{eq1.4}.\\
Using Lemma \ref{Lemma 12.4.2, 12.4.3, 12.4.4} (iii) combined with the fact that $u_{h_k}-u_Z\in W^{1,q}_{\#}(Y; \R^d)\cap L^{\infty}(Y; \R^d)$ and \eqref{eq1.4}, we conclude that 
    \begin{align}
        \widetilde{g}^q_\hom(tZ)&=\inf\left\{\int_Y g\left({x\over \e_{h}}, tZ+Du\right)dx \hspace{0.03cm}: \hspace{0.03cm} v\in W^{1,q}_{\#}(Y; \hspace{0.03cm} \R^d)\cap L^\infty(Y; \hspace{0.03cm} \R^d)    \right\}\notag\\
        &\leq \int_{Y} g\left({x\over \e_{h_k}}, tZ+tD(u_{h_k}(x)-u_Z(x))\right)dx<+\infty,\notag
    \end{align}
as desired.
\end{proof}

The next proposition shows an inequality for $G'$ similar to the one proved in Proposition \ref{Lemma 12.4.1}.

\begin{prop}
\label{Proposition 12.4.6}
    Let $g: \R^n\times \R^{d \times n} \to [0, +\infty]$  satisfy {\rm (H1)}-{\rm (H4)}. 
     Let $\widetilde{g}^q_\hom$ be defined by \eqref{ftildeCDA} and $q\in[1, +\infty]$. Then,
        \begin{equation}
        \notag
          \mathcal L^n(\Omega)   (\widetilde{g}^q_\hom)^{\rm ls} (Z) \leq G'(\Omega, u_Z),
        \end{equation}
    for any open and bounded subset $\Omega\subseteq \mathbb R^n$ and $Z\in\R^{d\times n}$.
\end{prop}

We skip the proof of Proposition \ref{Proposition 12.4.6}, since it follows from Lemma \ref{Lemma 12.4.2, 12.4.3, 12.4.4} and \ref{Lemma 12.4.5}, arguing as in \cite[Proposition 12.4.6]{CDA02}, recalling that \eqref{H40},  
crucial in Lemma \ref{Lemma 12.4.2, 12.4.3, 12.4.4}, is a consequence of \eqref{hpL} in $(H4)$, see Remark \ref{remark:existenceofcube}. \color{black} We only would like to point out that the set appearing in \eqref{ass-F'meno12} may appear a bit weird; on the other hand it serves as an intermediate step to prove Proposition \ref{Proposition 12.4.6} in the particular case first $\Omega = Y;$, in this case, \eqref{ass-F'meno12} turns to be a consequence of Lemma \ref{Lemma 12.4.2, 12.4.3, 12.4.4} item (ii).
\\
\smallskip

Finally, we observe that the values of $G'(\Omega, \cdot)$ and $G''(\Omega, \cdot)$ coincide on the space of linear functions, as stated in the next result. The proof is omitted, being entirely similar to the one of \cite[Proposition 12.4.7]{CDA02} and relying on Propositions \ref{Proposition 12.4.6} and \ref{Lemma 12.4.1}, taking into account the properties of inner regular envelopes of measures (see \cite[Chapter 15]{DM93}).

\begin{prop}
\label{Prop 12.4.7}
    Under the same assumptions of Proposition \ref{Proposition 12.4.6}, it holds that
       \begin{equation}
           \notag
           G'(\Omega, u_Z) = G''(\Omega, u_Z) = G'_{-}(\Omega, u_Z) = G''_{-}(\Omega, u_Z)= {\mathcal L}^n(\Omega) (\widetilde{g}^q_\hom)^{\rm ls} (Z),
       \end{equation}
    for any $\Omega\in{\color{blue} {\cal T}_0}$ and $Z\in\R^{d\times n}$.
\end{prop}
The next proposition states that the inner regular envelope of the functional $G'$ given by \eqref{defliminf} satisfies a blow-up condition. It generalizes \cite[Proposition 12.5.2]{CDA02} to the vectorial setting and the proof, following the same lines of \cite[Proposition 12.5.2]{CDA02}, is omitted.

  \begin{prop}
      \label{Prop 12.5.2}
       Let $g: \R^n\times \R^{d \times n} \to [0, +\infty]$  satisfy {\rm (H1)} and {\rm (H2)},  and let $q\in [1,+\infty]$. Let $G'_{-}$ be the inner regular envelope of the functional defined by \eqref{defliminf}. Then, 
           \begin{equation}
               \notag
               \limsup_{r\to 0} {1\over r^n} G'_{-}(Q_r(x_0), u) \geq 
                   G'_{-} (Q_1(x_0), u(x_0)+ Du(x_0)\cdot (\cdot-x_0)),
           \end{equation}
        for a.e. $x_0\in\R^n$ and for any $u\in W^{1,\infty}_\loc(\R^n; \R^d)$.
  \end{prop}

\subsection{Representation on Sobolev spaces}
The proof of Theorem \ref{Proposition 12.6.1} strongly relies on an integral representation result for unbounded functionals on Sobolev spaces. The aim of this subsection is to recall an abstract result which provides sufficient conditions to represent an unbounded functional $G$ depending on a bounded open subset $\Omega$ of  $\mathbb R^n$ and a function $u\in W^{1,p}_{\rm loc}(\R^n; \R^d)$, for $p\in[1,+\infty]$ as an integral. We state a generalization of \cite[Theorem 9.3.8]{CDA02} in the vectorial setting. Since its proof follows the same arguments, we skip it.
\par We start by providing the definition of weakly super-additivity and weakly sub-additivity (see \cite[Definition 2.6.5]{CDA02}).

\begin{definition}
Let $\mathcal{O}$ \color{blue} be the family of open subsets of $\mathbb R^n$ \color{black} and $\alpha: \mathcal{O} \rightarrow [0, + \infty].$ We say that $\alpha$ is:
\begin{itemize}
    \item[{\rm (i)}] {\sc weakly superadditive} if 
    \begin{equation*}
        \alpha (\Omega_1) + \alpha(\Omega_2)  \le \, \alpha(\Omega),
    \end{equation*}
for every $\Omega_1, \Omega_2, \Omega \in \mathcal{O}$ with $\Omega_1 \cap \Omega_2 = \emptyset,$ $\Omega_1 \cup \Omega_2 \subset \subset \Omega,$
    \item[{\rm (ii)}] {\sc weakly subadditive} if
       \begin{equation*}
           \alpha(\Omega) \le \alpha (\Omega_1) + \alpha(\Omega_2), 
       \end{equation*}
for every $\Omega, \Omega_1, \Omega_2 \in \mathcal{O}$ with $\Omega_1 \cap \Omega_2 = \emptyset,$ $\Omega \subset \subset  \Omega_1 \cup \Omega_2 $.
\end{itemize}
\end{definition}

For $p\in[1,+\infty]$, let $G$ be the functional defined by  
    \begin{equation}
    \label{(9.3.2)CDA}
    G: (\Omega, u) \in {\color{blue} \mathcal{T}_0} \times W^{1,p}_{\rm loc}(\mathbb{R}^n; \mathbb{R}^d) \mapsto G(\Omega, u) \in [0, + \infty],
    \end{equation}
and let $g_G$ be the function given by 
    \begin{equation}
    \label{(9.1.6)CDA}
     g_G: Z \in \mathbb{R}^{d\times n} \mapsto G(Y, u_Z) \in [0, + \infty],
    \end{equation}    
where $u_Z(x)\coloneqq Zx$, as above. 
We introduce the following set of assumptions:
   \begin{itemize}
       \item[(T1)]  invariance property for linear functions 
       \begin{equation}
        \notag
           G(\Omega; u_Z+c) = G(\Omega; u_Z)\qquad \mbox{for any } \Omega\in{\color{blue} {\cal T}_0}, 
       \end{equation}
       for any $\Omega\in{\color{blue} {\cal T}_0}$, $Z\in\R^{d\times n}$, and $ c\in\R^d$;
       \item[(T2)]  translation invariance 
           \begin{equation}
           \notag
           G(\Omega-x_0, T[x_0]u_Z) = G(\Omega, u_Z),
          \end{equation}
     for any $\Omega\in{\color{blue} {\cal T}_0}$, $Z\in\R^{d\times n}$ and $x_0\in\R^n$, where $T[x_0]u$ is the translate of $u$ defined by $T[x_0]u(x)\coloneqq u(x-x_0)$ (see Subsection \ref{not&pre});
     \item[(T3)] for every $u \in W^{1,\infty}_{\rm loc}(\mathbb{R}^n; \mathbb{R}^d),$ $G(\cdot, u)$ is increasing,  weakly superadditive and weakly subadditive; 
     \item[(T4)] the blow-up condition holds 
          \begin{equation}
          \notag
          \limsup_{r \to 0^+} \frac{1}{r^n} G(Q_r(x_0),u) \geq  G(Q_1(x_0), u(x_0) + Du(x_0)\cdot (\cdot-x_0)),
          \end{equation}
        for every $u \in W^{1,\infty}_{\rm loc}(\mathbb R^n;\mathbb R^d)$ and a.e. $x_0 \in \mathbb R^n$;
    \item[(T5)] for any $\Omega\in{\color{blue} {\cal T}_0}$,  $G(\Omega, \cdot)$ is $W^{1,p}(\Omega;\mathbb R^d)$-lower semicontinuous if $p \in  [1,+\infty)$, and $\cap _{q \in [1,+\infty)} W^{1,q}(\Omega;\mathbb R^d)$-lower semicontinuous if $p =+\infty$; 
    \item[(T6)]  for every $u \in W^{1,p}(\mathbb R^n;\mathbb R^d)$, $G(\cdot, u)$ is inner  regular;
    \item[(T7)] there exist $z_0 \in {\rm dom} g_G$, $r_0 > 0$ and a Radon positive measure $\mu$ on $\mathbb{R}^n$ such that $G(\Omega, u) \le \mu(\Omega)$ whenever $\Omega \in {\color{blue} {\cal T}_0}, u \in \color{blue}A_{\rm pw}\color{black}(\mathbb{R}^n; \mathbb{R}^d)$ with $Du(x) \in {\rm dom} g_G$ for a.e. $x \in \Omega$ and $\|u - u_{Z_0}\|_{W^{1, \infty}(\Omega; \mathbb{R}^d)} < r_0$; 
    \item[(T8)]  for every $\Omega \in {\color{blue} {\cal T}_0},$ $G(\Omega, \cdot)$  is convex; 
    \item[(T9)] there exist $z_0 \in \mathbb R^{d\times n}$ and  $\delta>0$ such that $B_\delta(z_0) \subseteq {\rm dom} g_G$.
    
   \end{itemize}

The next result shows that under the assumptions listed above, the unbounded functional $G$ may be represented as an integral whose energy density is given by $g_G$.

\begin{thm}
\label{teo9.3.8CDA}
Let $p \in (1, + \infty]$. Let $G$ be a functional as in \eqref{(9.3.2)CDA}, satisfying assumptions {\rm (T1)}-{\rm (T9)}, and let $g_G$ be given by \eqref{(9.1.6)CDA}. Then, \color{blue} it results that \color{black} $g_G$ \color{blue}is \color{black} convex and lower semicontinuous and the following representation
   \begin{equation}
       \notag
       G(\Omega, u) =\int_{\Omega} g_G(Du)dx
   \end{equation}
    holds for any $\Omega\in{\cal T}_0$ and $u\in W^{1,p}_{\loc}(\R^n; \R^d)$. 
\end{thm}
\begin{proof}[Proof] It follows along the lines of \cite[Theorems 9.3.7 and 9.3.8]{CDA02}, relying first on suitable estimates on $C^1(\mathbb R^n;\mathbb R^d)$ fields as in \cite[eq (9.2.1) and (9.2.2)]{CDA02}, \color{black} then on a standard density argument in $W^{1,q}(\Omega;\mathbb R^d)$ for any $q \in [1,+\infty)$ and on the lower semicontinuity of $G$ in the sense of ${\rm (T5)}$.  
\end{proof}

\subsection{The proof of Theorem \ref{Proposition 12.6.1}}

In this subsection, we provide the proof of Theorem \ref{Proposition 12.6.1} which relies on all the results presented in the previous subsections.

\begin{proof} [Proof of Theorem \ref{Proposition 12.6.1}]
    Let $\{\e_k\}_{k}$ be a strictly increasing sequence, then  \cite[Proposition 3.4.3]{CDA02},  ensures that there exists a further subsequence $\{\e_{k_j}\}_{j}$ such that 
        \begin{align}
            \notag
            &\sup\left\{\Gamma(L^\infty(A;\R^d))\mbox{-}\liminf_{j\to+\infty} G_{\e_{k_j}}(A, u) : A\subset\subset\Omega \right\}\notag\\
            &\quad = \sup\left\{\Gamma(L^\infty(A;\R^d))\mbox{-}\limsup_{j\to+\infty} G_{\e_{k_j}}(A, u) : A\subset\subset\Omega \right\}, \label{defmainthm}
        \end{align}
    for any $\Omega\in{\color{blue}{\cal T}_0}$ and $u\in L^\infty_\loc(\R^n; \R^d)$. 
    Fix $p\in (n, +\infty]$, so that the following embeddings $W^{1,p}_\loc(\R^n;\R^d)\subseteq C(\R^n;\R^d) \subseteq L^\infty_\loc(\R^n;\R^d)$ hold. For any $\Omega\in{\cal T}_0$, we define the functional $H(\Omega, \cdot):W^{1,p}_\loc(\R^n;\R^d)\mapsto [0,+\infty] $ as the common value of \eqref{defmainthm}. 
 
    Now, we show that the functional $H$ given by \eqref{defmainthm} satisfies assumptions (T1)-(T9). Indeed, (T1) is a consequence of the fact that the functionals  $G_{\e_{k_j}}$ as well as $H$  depend only on the gradient of the function $u$. \\
    (T2) follows from a generalization of \cite[Proposition 12.1.1]{CDA02} to the vectorial case, obtained with identical arguments. Proposition \ref{Proposition 12.2.1} combined with the fact that the functionals $G'$ and $G''$ are increasing set functions, for any $u\in L^\infty_\loc(\R^n; \R^d)$ implies (T3). To show (T4), first note that by definition of $H$ and standard results concerning $\Gamma$-limits, and their inner regular envelopes, we deduce that 
         \begin{align}
             &\limsup_{r\to 0} {\frac{1}{r^n}} H(Q_r(x_0), u)\notag\\
             &\quad = \limsup_{r\to 0} {1\over r^n} \sup\left\{ \Gamma(L^\infty(A; \R^d))\mbox{-}\limsup_{j\to+\infty} G_{\e_{k_j}}(A, u) : A\subset\subset Q_r(x_0) \right\}\notag\\
             &\quad = \limsup_{r\to 0} {1\over r^n} \sup\left\{ \Gamma(L^\infty(A; \R^d))\mbox{-}\liminf_{j\to+\infty} G_{\e_{k_j}}(A, u) : A\subset\subset Q_r(x_0) \right\}\notag\\
             &\quad \geq \limsup_{r\to 0} {1\over r^n} \sup\left\{ \Gamma(L^\infty(A; \R^d))\mbox{-}\liminf_{\e\to 0} G_{\e}(A, u) : A\subset\subset Q_r(x_0) \right\}\notag
             \end{align}
             \begin{align}
             &\quad\geq \limsup_{r\to 0} {1\over r^n} G'_- (Q_r(x_0), u)\notag\\
             &\quad\geq G'_-(Q_1(x_0), u(x_0)- Du(x_0)\cdot(\cdot-x_0)), \label{eq2mainthm}
         \end{align}
    where in the last inequality we have used Proposition \ref{Prop 12.5.2}. Since for any $\Omega\in{\color{blue} {\cal T}_0}$, $G'(\Omega, u+c) = G'(\Omega, u)$ and $G''(\Omega, u+c) = G''(\Omega, u)$ for any $u\in L^\infty_\loc(\R^n; \R^d)$ and $c\in\R^d$, an application of Proposition \ref{Prop 12.4.7} yields to 
         \begin{align}
             G'_-(Q_1(x_0), u(x_0)- Du(x_0)\cdot(\cdot-x_0)) &= G'_- (Q_1(x_0), - Du(x_0)\cdot(\cdot-x_0))\notag\\
             &= G''_- (Q_1(x_0), - Du(x_0)\cdot(\cdot-x_0))\notag\\
             &= G''_- (Q_1(x_0), u(x_0)- Du(x_0)\cdot(\cdot-x_0)).\notag
         \end{align}
    In particular, this implies that there exists  the inner regular envelope of the $\Gamma$-limit $G$, i.e., 
        \begin{align}
          &G_- (Q_1(x_0), u(x_0)- Du(x_0)\cdot(\cdot-x_0))= G'_- (Q_1(x_0), u(x_0)- Du(x_0)\cdot(\cdot-x_0))\notag\\
          &\quad = G''_-(Q_1(x_0), u(x_0)- Du(x_0)\cdot(\cdot-x_0)).\notag
        \end{align}
     Hence, using once again the properties of $\Gamma$-upper limits, from \eqref{eq2mainthm}, we get 
        \begin{align}
            &\limsup_{r\to 0} {1\over r^n} H(Q_r(x_0), u)\notag\\ 
             &\quad\geq G_- (Q_1(x_0), u(x_0)- Du(x_0)\cdot(\cdot-x_0))\notag\\
             &\quad=\sup\left\{ \Gamma(L^\infty(A; \R^d))\mbox{-}\limsup_{\e\to 0} G_{\e}(A, u(x_0)- Du(x_0)\cdot(\cdot-x_0)) : A\subset\subset Q_1(x_0) \right\}\notag\\
             &\quad\geq\sup\left\{ \Gamma(L^\infty(A; \R^d))\mbox{-}\limsup_{j\to +\infty} G_{\e_{k_j}}(A, u(x_0)- Du(x_0)\cdot(\cdot-x_0)) : A\subset\subset Q_1(x_0) \right\}\notag\\
             &\quad= H(Q_1(x_0), u(x_0)- Du(x_0)\cdot(\cdot-x_0)),\notag
        \end{align}
    which proves (T4).\\ 
    Assumption (T5) holds too. Indeed, given $\Omega\in {\color{blue} {\cal T}_0}$ and an open set with Lipschitz boundary $A$ such that $A\subset\subset \Omega$, it is known that  $\Gamma(L^\infty(A; \R^d))\mbox{-}\liminf_{j\to +\infty} G_{\e_{k_j}}(A, \cdot)$ and $\Gamma(L^\infty(A; \R^d))\mbox{-}\limsup_{j\to +\infty} G_{\e_{k_j}}(A, \cdot)$ are $W^{1,p}(\Omega; \R^d)$ (or $\cap_{q\in[1, +\infty)}W^{1,q}(\Omega; \R^d) $)-lower semicontinuous in $W^{1,p}_\loc(\R^n; \R^d)$. Therefore, $H(\Omega, \cdot)$ is also lower semicontinuous since it agrees with the last upper bound of the family of such functionals obtained letting $A$ vary with the above properties.\\
    Since, by definition, $H$ is sup of inner regular functionals, it follows that also (T6) is satisfied. To prove (T7), observe that, thanks to Proposition \ref{Prop 12.4.7}, we deduce that $g_G$ given by Theorem \ref{teo9.3.8CDA} agrees with $(\widetilde{g}^q_\hom)^{\rm ls}$. Moreover, due to (H3), there exists $\delta\in(0,1)$ such that $B_{2\delta}(0)\subseteq {\rm dom} \widetilde{g}^q_\hom \subseteq {\rm dom}\left(( \widetilde{g}^q_\hom)^{\rm ls}\right)$. Hence, choosing $z_0=0$ in (T7) and exploiting the properties of $\Gamma$-limits, as well as Proposition \ref{Prop 12.3.2}, we conclude that 
         \begin{align}
             H(\Omega, u)&\leq \sup\left\{ \Gamma(L^\infty(A; \R^d))\mbox{-}\limsup_{\e\to 0} G_\e(A, u) : A\subset\subset\Omega \right\}\notag\\
             &=\sup\left\{ G''(A, u) : A\subset\subset\Omega \right\}\notag\\
             &\leq c \mathcal L^n(\Omega),\notag
         \end{align}
    yielding the validity of (T7).  (T8) follows from the fact that, for any $q\in[1,+\infty]$, $\widetilde{g}^q_\hom$ \color{blue} is \color{black} a convex function. \\
    Moreover, the same consideration based on the convexity of $g(x, \cdot)$ and the properties inherited by $\widetilde g_{\rm hom}$, proved in Proposition \ref{prop:Proposition 12.1.3},  guarantee that $\widetilde g_{\rm hom}$ satisfies also (T9).
    Since the functional $H$ given by \eqref{defmainthm} satisfies all assumptions $(T1)$-$(T9)$, we can apply Theorem \ref{teo9.3.8CDA} so that
        \begin{equation}
            \label{intrepmainthm}
            H(\Omega, u) =\int_\Omega (\widetilde{g}^q_\hom)^{\rm ls} (Du)dx,
        \end{equation}
    for any $\Omega\in{\cal T}_0$ and $u\in W^{1,p}_\loc (\R^n; \R^d)$, where we have used the fact that, due to Proposition \ref{Prop 12.4.7}, $g_G$ of Theorem \ref{teo9.3.8CDA} agrees with $( \widetilde{g}^q_\hom)^{\rm ls}$. Since the representation \eqref{intrepmainthm} holds for any subsequence $\{\e_{k_j}\}_j$ of $\{\e_k\}_k$, by Uryshon's property, we conclude that 
        \begin{equation}
            \notag
            G'_{-}(\Omega, u) = G''_{-}(\Omega, u) = \int_\Omega  (\widetilde{g}^q_\hom)^{\rm ls} (Du)dx,
        \end{equation}
    for any $\Omega\in{\cal T}_0$ and $u\in W^{1,p}_\loc (\R^n; \R^d)$, as desired.

   The conclusion is achieved by repeating word by word the proof of \cite[Proposition 2.7.4]{CDA02},  exploiting the convexity of $\Omega$.
   
\end{proof}

\section{Homogenization}

\label{cinque}

In this section, we \color{blue} provide a homogenization result for the family of energies $\{F_\e\}_{\e}$ given by \eqref{def:supfunctional}, arguing directly in terms of supremal functionals, exploiting the results on the homogenization of unbounded integral functionals obtained in the previous section.
\color{black}

\begin{thm}
\label{Theorem5.1}
    Let $\Omega$ be a bounded convex open subset of $\R^n$. Let $f: \R^n\times \R^{d\times n}\to[0,+\infty)$ be a Borel function, $1$-periodic in the first variable satisfying the growth conditions \eqref{growthconditions}. For any $\varepsilon >0$, let $F_\e$ be the supremal functional defined by \eqref{def:supfunctional} and let $F^{\color{blue} \hom}$ be the functional defined by 
          \begin{equation}
              \notag
              F^{\color{blue}\hom\color{black}}(u):= \underset{x\in \Omega}{\mbox{{\rm ess}-{\rm sup}}}\hspace{0.1cm} \widetilde{f}_\hom (Du(x)),
          \end{equation}
   where $\widetilde{f}_\hom$ is given by \eqref{characterizationftildehom}. We have that
   \begin{itemize}
       \item [\rm (i)] For any $u\in W^{1,\infty}(\Omega; \hspace{0.03cm}\R^d)$ and for any $u_\e\in W^{1,\infty}(\Omega; \hspace{0.03cm}\R^d)$ such that $u_\e$ uniformly converges to $u$, then 
           \begin{equation}
               \label{Gammaliminfine}
               \liminf_{\e\to 0} F_\e(u_\e) \geq F^{\color{blue}\hom} (u).
           \end{equation}
    \item[\rm (ii)] Assume that $f(x,\cdot)$   
    is level convex and continuous for every $x\in Y$, \color{black} and there exist matrices $\{A_i\}_{i=1}^{nd}\subset\R^{d\times n}$ vertices of a cube $Q \ni 0$, such that
        \begin{equation}
            \notag
            f(x, A_i)\coloneqq \min_{Z\in\R^{d\times n}}f(x, Z),
        \end{equation}
        for every $x\in\R^n$.
     Then, for any $u\in W^{1,\infty}(\Omega; \hspace{0.03cm}\R^d)$, there exists a sequence $\{u_\e\}_\varepsilon\subset W^{1,\infty}(\Omega; \hspace{0.03cm}\R^d)$ such that $u_\e$ uniformly converges to $u$ and 
              \begin{equation}
               \notag
               \lim_{\e\to 0} F_\e(u_\e) = F{\color{blue}^{\hom}} (u).
           \end{equation}
   \end{itemize}
\end{thm}

The proof of the liminf inequality relies on the results obtained by 
 the $L^p$-approximation. Instead, the proof of the limsup inequality is trickier and is based on a generalization of \cite[Chapter 12]{CDA02} in the vectorial case.  This is the reason why we require the convexity of the domain $\Omega$, 
 and the level convexity and the upper semicontinuity of $f$. 
  Once such a generalization is obtained (cf. Section \ref{secHomunb}), the proof of the limsup inequality is an adaption of the techniques developed in \cite[Proposition 4.4]{BGP04}. However, for the readers' convenience, we provide a complete proof.

 We require that $f(x,\cdot)$ is a continuous function since in view of Remark \ref{rmk:cellformula}  we can specialize the homogenized energy density as follows
             \begin{equation}
                 \notag
                 \widetilde{f}_{\hom} (Z) = \inf\left\{  \underset{x\in Y}{\mbox{{\rm ess}-{\rm sup}}}\hspace{0.1cm}f(x, Z+ Du(x))  : u\in W^{1,\infty}_{\#}(Y; \R^d ) \right\}.
             \end{equation}

\begin{proof}
\begin{itemize}
    \item [\rm (i)] We prove the $\Gamma$-liminf inequality. To that end, without loss of generality, we may assume that $\sup_{\e> 0} F_\e(u_\e)<\infty$. Since $f$ satisfied the growth conditions \eqref{growthconditions}, we get the homogenization result proven in \eqref{def:Fhomp} and \eqref{def:fhomp} for the functionals $\{F_{p, \e}\}_{p,\e}$ introduced in \eqref{def:Functionalpeps}. In particular, for any $p>1$,
        \begin{align}
            \left (\int_{\Omega}f^\hom_p(Du(x))dx \right )^{1/p} &\leq \liminf_{\e\to 0} \left( \int_{\Omega}f^p\left({x\over \e}, Du_{\e}(x)\right)dx  \right)^{1/p}\notag\\
            & \leq \left({\cal L}^n(\Omega)\right)^{1/p} \liminf_{\e\to 0} \esssup \hspace{0.1cm} f\left({x\over \e}, Du_{\e}(x)\right)\notag.
        \end{align}
    Thanks to Lemma \ref{lemma:convergenceofhomdensity}, we pass into the limit as $p\to+\infty$ obtaining that 
        \begin{align}
            \esssup\hspace{0.1cm}  \widetilde{f}_\hom (Du(x)) & =\lim_{p\to+\infty} \left (\int_{\Omega}f^\hom_p(Du(x))dx \right )^{1/p} \nonumber\\
            & \leq  \liminf_{\e\to 0} \esssup \hspace{0.1cm} f\left({x\over \e}, Du_{\e}(x)\right)\notag,
        \end{align}
which shows \eqref{Gammaliminfine}.

\item[\rm (ii)] The proof of the $\Gamma$-limsup inequality relies on the homogenization Theorem \ref{Proposition 12.6.1} proved in Section \ref{secHomunb}.  Fix $\overline{u}\in W^{1,\infty}(\Omega; \R^d)$ and set $M\coloneqq \esssup\hspace{0.01cm}\widetilde{f}_{\hom} (D\overline{u}(x))$. We aim at finding a sequence $\{u_\e\}_\e\subset W^{1,\infty}(\Omega; \R^d)$ such that $u_\e\to\overline{u}$ in $L^\infty(\Omega; \R^d)$ as $\e\to 0$ and 
      \begin{equation*}
          \limsup_{\e\to 0} F_\e(u_\e) \leq M.
      \end{equation*}
To that end, for any $x\in\Omega$, we introduce the sets $C(x)$ and $C_\infty$ defined by
    \begin{equation}
        \notag
        C(x)\coloneqq \{Z\in\R^{d\times n} : f(x, Z)\leq M  \}, \qquad C_\infty\coloneqq\{Z\in\R^{d\times n} : \widetilde{f}_\hom(Z)\leq M \}.
    \end{equation}
In view of the assumptions on $f$, it \color{blue}results \color{black} that the set $C(x)$ is measurable, $1$-periodic and convex. {\color{blue} Recall that a set $C\subset \R^{d\times n}$ is $1$-periodic if its indicator function $1_C$ is $1$-periodic}. Since $A_i$ \color{blue} are \color{black} minimum points of $f(x, \cdot)$, we deduce that $A_i\in C(x)$, for any $i=1,\dots, dn$. All these assumptions permit us to apply Theorem \ref{Proposition 12.6.1} to the indicator function $1_{C(x)}$ obtaining
     \begin{equation}\label{eq0.01}
         \Gamma(L^\infty)\mbox{-}\lim_{\e\to 0} \int_\Omega 1_{C\left( {x\over \e}\right)}(Du(x))dx =\int_\Omega \widetilde{1}^\infty_{\hom}(Du(x))dx,
     \end{equation}
where the homogenized energy density $\widetilde{1}^\infty_{\hom}$ is given by the cell formula 
    \begin{equation}
        \notag
        \widetilde{1}^\infty_{\hom}(Z)= \inf \left\{ \int_Y 1_{C(y)}(Z+Dv(y))dy :  v\in W^{1,\infty}_{\#}(Y; \R^d)  \right\}.
    \end{equation}
Thanks to Proposition \ref{prop:Proposition 12.1.3} applied to $p=q=+\infty$, the infimum is actually achieved.
To conclude the proof, it remains to prove that 
    \begin{equation}
        \notag
        \widetilde{1}^\infty_{\hom} (Z) = 1_{C_\infty}(Z) \qquad\mbox{for any } Z\in\R^{d\times n}.
    \end{equation}
It is enough to show that 
   \begin{equation}
       \notag
        \widetilde{1}^\infty_{\hom} (Z) =0\quad  \Leftrightarrow \quad 1_{C_\infty}(Z)=0.
   \end{equation}
Indeed, 
   \begin{align}
       \widetilde{1}^\infty_{\hom} (Z) =0  &\Leftrightarrow \mbox{there exists } v\in W^{1,\infty}_{\#}(Y; \R^d) : \int_Y 1_{C(y)}(Z+Dv(y))dy =0\notag\\
       &\Leftrightarrow \mbox{there exists } v\in W^{1,\infty}_{\#}(Y; \R^d) :  f(x, Z+ Dv(x)) \leq M \quad \mbox{for a.e. } x\in\Omega\notag\\
       &\Leftrightarrow \mbox{there exists } v\in W^{1,\infty}_{\#}(Y; \R^d) : \esssup \hspace{0.01cm} f(x, Z+Dv(x)) \leq M. \label{0.01}
   \end{align}
Thanks to lower semicontinuity and level convexity of $f$, the supremal functional $\mbox{ess-sup}_{x\in Y}f(x, Z+Dv(x))$ \color{blue} is \color{black} lower semicontinuous and coercive in $W^{1, \infty}_{\#}(Y; \R^d)$ and thus the infimum in the definition of $\widetilde{f}_{\hom}$ is achieved. 
Then, the last condition in \eqref{0.01} is equivalent to $\widetilde{f}_\hom(Z)\leq M$., i.e., $1_{C_\infty}(Z)=0$. Thanks to \eqref{eq0.01}, there exists a sequence $\{u_\e\}_\e$ of functions in $W^{1,\infty}(\Omega; \R^d)$ such that $u_\e$ uniformly converges to $\overline{u}$ as $\e\to 0$ and 
     \begin{equation}
         \notag
         \limsup_{\e\to 0} \int_\Omega 1_{C\left({x\over \e} \right)}(Du_\e(x))dx \leq \int_\Omega 1_{C_\infty} (D\overline{u}(x))dx =0.
     \end{equation}
In particular, there exists $\e_0>0$ such that, for any $\e\leq\e_0$, 
     \begin{equation}
         \notag
         \int_\Omega 1_{C\left({x\over \e} \right)}(Du_\e(x))dx =0.
     \end{equation}
In other words, 
    \begin{equation}
        \notag
        1_{C\left({x\over \e} \right)}(Du_\e(x)) =0  \quad \mbox{for a.e. } x\in\Omega.
    \end{equation}
In view of the definition of $C(x)$, it follows that $\esssup\hspace{0.01cm} f({x\over \e}, Du_\e(x))\leq M$ which, in turn, implies that $F_\e(u_\e)\leq M$ and 
    \begin{equation}
        \notag
        \limsup_{\e\to 0} F_\e(u_\e) \leq F{\color{blue}^{\hom}} (\overline{u}),
    \end{equation}
as desired.
\end{itemize}
\end{proof}

\color{blue}
\appendix
\section{Appendix}

For the readers's convenience, we state a $\Gamma$-convergence result which, as a particular case, provides a proof of the upper horizontal arrow in our diagram, only considering $\varepsilon$-dependence. The proof can be found in \cite[Section 5]{PZ20}.

\begin{thm}\label{curlcase2} Let $\Omega\subseteq \R^n$ be a bounded open set with Lipschitz boundary.
	Let $f:\Omega\times \mathbb R^{d\times n}\to[0,+\infty)$ be a $\mathcal L^n(\Omega) \otimes \mathcal B(\R^{d\times n})$-measurable function satisfying \eqref{growthconditions}.
	For every $p\geq 1$ let  
	$F_p: C(\bar \Omega;\R^d)\to [0,+\infty)$  be the functional given by 
	\begin{equation}
\notag
 F_p(u):=\left\{\begin {array}{cl}
	\displaystyle \left( \int_{\Omega} f^p(x,\nabla u(x))dx \right)^{1/p}
	&  \hbox{if } \, u\in W^{1,p}(\Omega;\R^d),\\
	+\infty  & \hbox{otherwise}.
\end{array}\right.
\end{equation}  
\noindent Then there exists a $\mathcal L^n(\Omega) \otimes \mathcal B(\R^{d\times n})$-measurable function $f_{\infty}:\Omega\times  \mathbb R^{d\times n}\to [0,+\infty)$ such that    $\{F_p\}_{p\geq 1}$ $\Gamma(L^{\infty})$-converges, as $p\to\infty$, to the functional $\bar F:C(\bar \Omega;\R^d)\to \overline \R$ defined as

 \begin{equation}\notag
\bar F(u):=\left\{\begin {array}{cl} \displaystyle {\rm ess}\sup_{x\in \Omega}
f_{\infty}(x,\nabla u(x))
&  \hbox{if } \, u\in W^{1,\infty}(\Omega;\R^d),\\
+\infty  & \hbox{otherwise.}
\end{array}\right.
\end{equation}
Moreover for a.e. $x\in \Omega$, $f_{\infty}(x,\cdot)$  is a strong Morrey quasiconvex function satisfying 
\begin{equation}
 \notag
f_{\infty}(x,\cdot)\geq Q_{\infty}f(x,\cdot):=\sup_{p\geq 1} (Qf^p)^{1/ p}(x,\cdot),  
\end{equation}
where $Qf^p(x,\cdot):=Q(f^p)(x,\cdot)$ stands for the quasiconvex envelope of $f^p(x,\cdot)$.
\end{thm}

We observe that if $f$ is continuous and $f(x,\cdot)$ is level convex for every $x \in \Omega$, then $f_\infty= f$. We refer to \cite[Remark 5.1]{PZ20} for more comments. Indeed this latter observation allows us to prove that the diagram in the introduction is commutative.

\bigskip

Next, we comment on the equality \eqref{eqhf}, recalling the arguments in \cite[proof of Theorem 2.2]{PZ20}. 
Since for every $p\geq 1$, $\widetilde{ f^p}$ is a  Carath\'edory function, we deduce that 
\begin{equation}\label{pq} 
(\widetilde{ f^p})^{1/ p}(x,Z)\leq (\widetilde{ f^q})^{1/ q}(x,Z)\end{equation} for a.e. $x\in \Omega$ and $Z\in  \R^{d\times n}$.
Then, set 
\begin{equation}\label{finftydef}f_{\infty}(x,Z):= \sup_{k\geq 1} (\widetilde{ f^k})^{ 1/ k}(x,Z),
\end{equation} it results that 
$f_{\infty}$ is $\mathcal L^n(\Omega) \otimes \mathcal B(\R^{d\times n})$-measurable function, being the countable supremum of Carath\'eodory functions, 
and for a.e. $x\in \Omega$ and $Z\in  \R^{d\times n}$,
\begin{equation} \notag
f_{\infty}(x,Z)=\lim_{k\to\infty}  \big(\widetilde{ f^k}\big)^{1/k}(x,Z).
\end{equation}
Hence, taking  into account \eqref{growthconditions} and \eqref{finftydef}, it results
\begin{equation}\notag
 \alpha |Z|\leq f_{\infty}(x,Z) \leq \beta (1+ |Z|).
\end{equation}

Moreover, as proven in \cite[Proposition 5.1]{PZ20}  for a.e. fixed $x\in \Omega$, the function $f_{\infty}(x,\cdot)$ is strong Morrey quasiconvex.  

Finally, it is easy to show that if  $\{p_k\}_k$ is a divergent sequence, the $\mathcal L^n(\Omega) \otimes \mathcal B(\R^{d\times n})$-measurable function $h_{\infty}:\Omega\times \R^{d\times n}\to [0,\infty]$ defined by $$h_{\infty}(x,Z):=\sup_{k\geq 1} (\widetilde{ f^{p_k}})^{ 1/ {p_k}}(x,Z),$$ satisfies 

\begin{equation*}\esssup_{\Omega} h_{\infty}(x,Du(x)) =\esssup_{\Omega} f_{\infty}(x,Du(x)) \quad \forall u\in  W^{1,\infty}(\Omega,\mathbb R^d).
\end{equation*}
Indeed, for every fixed $u\in W^{1,\infty}(\Omega,\R^d)$ and  for every  fixed $\varepsilon>0$ there exists $\bar k\in\N$ such that
\begin{align} \label{supdis} \esssup_{\Omega} h_{\infty}(x,Du(x)) &=\sup_{k\geq 1} \esssup_{\Omega} (\widetilde{ f^{p_k})}^{ 1/ {p_k}} (x,Du(x))\\
&\leq \esssup_{\Omega} (\widetilde{ f^{p_{\bar k}}})^{ 1/ {p_{\bar k}}}(x,Du(x)) +\varepsilon. \nonumber \end{align}
Then,  by \eqref{pq} there exists a measurable set $\Omega'\subseteq \Omega$ such that ${\mathcal L}^n(\Omega\setminus \Omega')=0$ and 
 $$(\widetilde{ f^{p_{\bar k}}})^{1/ {p_{\bar k} }}(x,\xi)\leq (\widetilde{ f^k})^{1/  k}(x,\xi),$$  for every $k\geq p_{\bar k}$,
 for every  $x\in \Omega'$ and $\xi\in  \R^{d\times n}$.
 In particular, \eqref{supdis} and \eqref{finftydef} imply
$$\esssup_{\Omega} h_{\infty}(x,Du(x))\leq \esssup_{\Omega} (\widetilde{ f^k }(x,Du(x)))^{ \frac 1 k }+\varepsilon \leq \esssup_{\Omega}f_{\infty}(x,Du(x))+\varepsilon .$$
By sending $\varepsilon$ to $0$ we get that 
$$\esssup_{\Omega} h_{\infty}(x,Du(x))\leq \esssup_{\Omega}f_{\infty}(x,Du(x)) .$$

The proof of the converse inequality is analogous.

\vspace{1.3cm}

\color{black}
{\bf Acknowledgements}

L. D. acknowledges the support of the Austrian Science Fund (FWF) through grants F65, V662, Y1292. M. E.  has been partially supported by PRIN 2020 ``Mathematics for industry 4.0 (Math4I4)'' (coordinator P. Ciarletta).
E. Z. acknoledges the received support through Sapienza Progetti d'Ateneo 2022 piccoli ``Asymptotic Analysis for composites, fractured materials and with defects '', and through PRIN 2022 `` Mathematical Modelling of Heterogeneous Systems'' (coordinator E. N. M. Cirillo), CUP B53D23009360006.
The authors are members of INdAM-GNAMPA, whose support is gratefully ackowledged through GNAMPA Project 2023 
  ``Prospettive nelle scienze dei materiali: modelli variazionali, analisi asintotica ed omogeneizzazione''.\\
  {\color{blue} We thank the anonymous referees for their valuable comments that improved the presentation of the paper.}

\end{document}